\documentclass[11pt]{article}
\usepackage[utf8]{inputenc}
\usepackage{lmodern}
\usepackage[T1]{fontenc}
\usepackage[top=1in, bottom=1.in, left=1in, right=1in]{geometry}
\usepackage{graphicx}
\usepackage{longtable}
\usepackage{float}
\usepackage{wrapfig}
\usepackage{rotating}
\usepackage{tikz}
\usetikzlibrary{shadows.blur}
\usetikzlibrary{arrows}
\usetikzlibrary{positioning}
\usepackage[normalem]{ulem}
\usepackage{amsmath}
\usepackage{textcomp}
\usepackage{marvosym}
\usepackage{wasysym}
\usepackage{amssymb}
\usepackage{amsmath}
\usepackage{nicefrac}
\usepackage[theorems, skins]{tcolorbox}
\usepackage[version=3]{mhchem}
\usepackage[numbers]{natbib}
\usepackage{url}
\usepackage[strings]{underscore}
\usepackage[linktocpage,pdfstartview=FitH,colorlinks,
linkcolor=blue,anchorcolor=blue,
citecolor=blue,filecolor=blue,menucolor=blue,urlcolor=blue]{hyperref}
\usepackage{attachfile}
\usepackage{setspace}
\usepackage{amsmath}
\usepackage{mathtools}
\usepackage{amsthm}
\usepackage{stmaryrd}
\usepackage{fancyheadings}
\usepackage{algorithm}
\usepackage{algpseudocode}
\usepackage[pdf]{graphviz}
\usepackage{tcolorbox}
\newtheorem{theorem}{Theorem}
\newtheorem{lemma}[theorem]{Lemma}
\newtheorem{proposition}[theorem]{Proposition}
\newtheorem{corollary}[theorem]{Corollary}
\theoremstyle{definition}
\newtheorem{definition}[theorem]{Definition}
\newtheorem{fact}[theorem]{Fact}

\newtheorem{remark}[theorem]{Remark}
\newtheorem{example}[theorem]{Example}

\newcommand{\TODO}[1]{\textcolor{red}{(TODO: #1)}}
\newcommand{\rpart}{\mathop{\mathrm{Re}}}
\newcommand{\ipart}{\mathop{\mathrm{Im}}}
\newcommand{\R}{\mathbb{R}}
\newcommand{\C}{\mathbb{C}}
\newcommand{\Z}{\mathbb{Z}}

\newcommand{\T}{\textsc{T}}
\newcommand{\pinv}{+}
\newcommand{\winv}{-}
\newcommand{\ha}{\textnormal{\textsc{h}}}
\newcommand{\sconj}{*}
\newcommand{\adj}{*}

\newcommand{\hadprod}{\bullet}
\newcommand{\faceprod}{\Delta}

\newcommand{\tprod}{\star_{\text{t}}}
\newcommand{\Mprod}{\star_{\matM}}
\newcommand{\Mleq}{\leq_{\matM}}
\newcommand{\Mgeq}{\geq_{\matM}}
\newcommand{\unitM}{\ve_{\matM}}
\newcommand{\KM}{{\mathbb{K}}_{\matM}}
\newcommand{\Mdot}{{\cdot}_{\matM}}
\renewcommand{\vec}[1]{\mathbf{\boldsymbol{#1}}}
\newcommand{\mat}[1]{\mathbf{#1}}
\newcommand{\ten}[1]{\mathcal{#1}}
\newcommand{\inner}[2]{\left\langle #1,#2 \right\rangle}
\newcommand{\dotprod}[2]{\left\langle #1,#2 \right\rangle_F}
\newcommand{\Xnorm}[2]{\left\| #1 \right\|_{#2}}
\newcommand{\Fnorm}[1]{\Xnorm{#1}{F}}

\newcommand{\XnormS}[2]{\left\| #1 \right\|^2_{#2}}
\newcommand{\FnormS}[1]{\XnormS{#1}{F}}

\newcommand{\lrapprox}[2]{\left[ #1 \right]_{#2}}

\newcommand{\rank}{\mathop{\mathrm{rank}}}
\renewcommand{\dim}{\mathop{\mathrm{dim}}}

\newcommand{\Mrank}{\mathop{\Mprod\mathrm{-rank}}}
\newcommand{\Mmultirank}{\mathop{\Mprod\mathrm{-multirank}}}
\newcommand{\squeeze}{\mathop{\mathrm{squeeze}}}
\newcommand{\twist}{\mathop{\mathrm{twist}}}

\newcommand{\spn}{\mathop{\mathrm{span}}}
\newcommand{\diag}{\mathop{\mathrm{diag}}}
\newcommand{\circmatrix}{\mathop{\mathrm{circ}}}
\newcommand{\range}{\mathop{\mathrm{range}}}
\newcommand{\nullsp}{\mathop{\mathrm{null}}}

\newcommand{\nmodeprod}[1]{\times_{#1}}

\newcommand{\sj}{\mathrm{j}}
\newcommand{\si}{\mathrm{i}}
\newcommand{\va}{\vec{a}}
\newcommand{\vb}{\vec{b}}
\newcommand{\vc}{\vec{c}}

\newcommand{\ve}{\vec{e}}
\newcommand{\vn}{\vec{n}}
\newcommand{\vi}{\vec{i}}
\newcommand{\vp}{\vec{p}}
\newcommand{\vr}{\vec{r}}
\newcommand{\vs}{\vec{s}}

\newcommand{\vu}{\vec{u}}
\newcommand{\vv}{\vec{v}}
\newcommand{\vw}{\vec{w}}
\newcommand{\vx}{\vec{x}}
\newcommand{\vy}{\vec{y}}
\newcommand{\vz}{\vec{z}}
\newcommand{\vzero}{\vec{0}}

\newcommand{\vbeta}{\vec{\beta}}

\newcommand{\vsigma}{\vec{\sigma}}

\newcommand{\matA}{\mat{A}}
\newcommand{\matB}{\mat{B}}
\newcommand{\matC}{\mat{C}}
\newcommand{\matD}{\mat{D}}
\newcommand{\matE}{\mat{E}}
\newcommand{\matF}{\mat{F}}

\newcommand{\matM}{\mat{M}}
\newcommand{\matN}{\mat{N}}

\newcommand{\matQ}{\mat{Q}}
\newcommand{\matR}{\mat{R}}
\newcommand{\matS}{\mat{S}}

\newcommand{\matI}{\mat{I}}
\newcommand{\matU}{\mat{U}}
\newcommand{\matV}{\mat{V}}
\newcommand{\matW}{\mat{W}}
\newcommand{\matX}{\mat{X}}
\newcommand{\matY}{\mat{Y}}
\newcommand{\matZ}{\mat{Z}}
\newcommand{\matZero}{\mat{0}}

\newcommand{\matLambda}{\mat{\Lambda}}

\newcommand{\tenA}{\ten{A}}
\newcommand{\tenB}{\ten{B}}
\newcommand{\tenC}{\ten{C}}
\newcommand{\tenD}{\ten{D}}

\newcommand{\tenI}{\ten{I}}
\newcommand{\tenU}{\ten{U}}
\newcommand{\tenV}{\ten{V}}

\newcommand{\tenS}{\ten{S}}
\newcommand{\tenX}{\ten{X}}

\newcommand{\Rep}{\mathrm{Rep}}

 \author{Haim Avron
\thanks{School of Mathematical Sciences, Raymond and Beverly Sackler Faculty of Exact Sciences, Tel Aviv University, Tel Aviv 6997801, Israel.}
\and
Uria Mor$^*$
}
\date{\today}
\title{Demystifying Tubal Tensor Algebra}
\begin{document}

\label{bibliographystyle link}
\bibliographystyle{plain}

\maketitle
\begin{abstract}
Developed in a series of seminal papers in the early 2010s, the tubal tensor framework provides a clean and effective algebraic setting for tensor computations, supporting matrix-mimetic features such as a tensor Singular Value Decomposition and Eckart–Young-like optimality results. It has proven to be a powerful tool for analyzing inherently multilinear data arising in hyperspectral imaging, medical imaging, neural dynamics, scientific simulations, and more.
At the heart of tubal tensor algebra lies a special tensor-tensor product: originally the t-product, later generalized into a full family of products via the $\Mprod$-product. Though initially defined through the multiplication of a block-circulant unfolding of one tensor by a matricization of another, it was soon observed that the t-product can be interpreted as standard matrix multiplication where the scalars are tubes—i.e., real vectors twisted ``inward.'' Yet, a fundamental question remains: why is this the ``right'' way to define a tensor-tensor product in the tubal setting?
In this paper, we show that the t-product and its $\Mprod$ generalization arise naturally when viewing third-order tensors as matrices of tubes, together with a small set of desired algebraic properties. Furthermore, we prove that the $\Mprod$-product is, in fact, the only way to define a tubal product satisfying these properties. Thus, while partly expository in nature — aimed at presenting the foundations of tubal tensor algebra in a cohesive and accessible way — this paper also addresses theoretical gaps in the tubal tensor framework, proves  new results, and provides justification for the tubal tensor framework central constructions, 
thereby shedding new light on it.
\end{abstract}

\section{Introduction}
\label{sec:introduction}

Across scientific computing, data science, machine learning, and many other fields, the de-facto standard way to organize numerical data and work with it algebraically, is using vectors and matrices. Matrix algebra provides a well-founded mathematical framework for working with data, and provides useful tools like the SVD decomposition, and associated powerful results like the Eckart-Young Theorem. Even when data is multilinear or multiway, it is common to eschew dimensional integrity, and force data into the form of matrices by flattening it, all for the sake of using linear algebra tools in the analysis.  

However, starting with the classical works of Hitchcock, Cattel, Tucker, Kruskal and Harshman~\cite{hitchcock-1927-expres-tensor,hitchcock-1928-multip-invar,cattell-1944-paral-propor,tucker-1963-implic,harshman-1970-parafac,kruskal-1977-three-way-array}, and much more profoundly in the last couple of decades, there is a push towards using representations and computational frameworks that preserve the dimensional integrity of the data, i.e., using tensor-based methods. Central to tensor-based methods are tensor decompositions that aim to generalize matrix SVD~\cite{kolda-2009-tensor-decom-applic}.
Unfortunately, there is more than one logical way to generalize matrix SVD. Even more unfortunate is the fact that computing many such decompositions is NP-hard~\cite{hillar-2013-most-tensor}, and the lack of Eckart-Young-like results that guarantee the utility of rank truncations. 

A notable exception to the above is the {\em tubal tensor framework}, introduced in a series of seminal papers during the 2010s~\cite{kilmer-2011-factor-strat,braman-2010-third-order,kilmer-2013-third-order,kernfeld-2015-tensor-tensor}\footnote{We remark that name ``tubal tensor algebra'' is not a standard and universally acceptable name for this algebra. In fact, there does not seem to be a standard and universally acceptable name for it. We believe the present paper makes it evident why the name ``tubal tensor algebra'' is an appropriate name for this algebra.}. 
It provides a clean algebraic framework for third-order tensor, which is {\em matrix-mimetic}, i.e., most concepts, decompositions, algorithms, results, etc. carry over to tensors in the algebra after some simple adaptions. Importantly, tubal tensor algebra supports a tensor SVD decomposition that is simple to compute~\cite{kilmer-2011-factor-strat}, and for which Eckart-Young-like results hold~\cite{kilmer-2021-tensor-tensor}.

Central to tubal tensor algebra is a closed product operation between two third-order tensors that preserves the order of a tensor: initially the t-product~\cite{kilmer-2011-factor-strat}, later generalized to the $\Mprod$-product~\cite{kernfeld-2015-tensor-tensor,kilmer-2021-tensor-tensor} (here $\matM$ is an invertible matrix that parameterizes the product; see Section~\ref{sec:org9537a00} for formal definitions). The utility of these tensor products is undeniable: they lead to the aforementioned elegant and powerful tensor algebra, and their applicability has been demonstrated in many works, e.g., in producing tensor compressions superior to ones based on other tensor decompositions,  or matricization~\cite{kilmer-2021-tensor-tensor}. 
At the same time, a fundamental question remains: why is $\Mprod$ a ``right'' way to define a tensor-tensor product?  



To answer this question, we turn to the foundations of the tubal tensor framework. One way to view tubal tensor algebra is via a very simple, and in a sense natural, precept:
\begin{center}
\begin{tcolorbox}[colback=red!5!white,colframe=red!75!black,halign=center,hbox]
\emph{View a third-order tensor as a matrix of vectors (called tubes).}
\end{tcolorbox}
\end{center}
In tubal tensor algebra, a tensor is a matrix of tubes, where each tube is a vector. Thus, the resulting algebra is matrix-mimetic by construction. However, for matrix-mimeticity to actually work it is crucial to define operations between tubes correctly. 


The {\em tubal precept}, illustrated graphically in Figure \ref{fig:tensorastubes}, appeared in the literature on tubal tensor algebra from its onset~\cite{Kilmer2008TR,braman-2010-third-order, gleich-2012-power-arnol}, but essentially as an implication the definition of tubal tensor-tensor product formed using terms of block-circulant matrices along with technical tensor-algebraic operations like mode matrix-tensor product, squeezing and twisting. 
An explicit formulation of the tubal precept was made in  \cite{kernfeld-2015-tensor-tensor}, which also established a connection to the tensor-tensor product. However, the tubal precept was still considered as a ``micro'' view of the tensor product, and not as the foundation on which the algebra is built.

\begin{figure}[htbp]
\centering
{\begin{tikzpicture}
    \def\width{0.35}    
    \def\height{\width}   
    \def\depth{2.3*\width}   
    \def\scl{0.8} 
    \def\spacing{.45}     

    \foreach \x in {4} {
    \begin{scope}[shift={(\x*\spacing,0)}]
        \foreach \y in {0,1,2,3} {
            \begin{scope}[shift={( \y * \spacing,0)}]
                \fill[ opacity=0.0, blur shadow={shadow blur steps=10,shadow scale = .82, shadow yshift=\y*\spacing, shadow blur radius=1.5ex}] (\width, 0) -- ++(\width, 0) -- ++(\depth*\scl, \depth) -- ++(-\width, 0) -- cycle;
            \end{scope}
        }
    \end{scope}
    }
    
    \foreach \x in {0,1,2,3,4} {
    \begin{scope}[shift={(\x*\spacing,0)}]
        \foreach \y in {0,1,2,3} {
            \begin{scope}[shift={(0, \y * \spacing)}]
                \fill[blue!60] (0, 0) -- ++(\width, 0) -- ++(0, \height) -- ++(-\width, 0) -- cycle;
                
                \fill[blue!40] (0, \height) -- ++(\width, 0) -- ++(\depth*\scl, \depth) -- ++(-\width, 0) -- cycle;

                \fill[blue!20] (\width, 0) -- ++(0, \height) -- ++(\depth*\scl, \depth) -- ++(0,-\height) -- cycle;

                \draw[thick] (0, 0) -- ++(\width, 0) -- ++(\depth*\scl, \depth) -- ++(0, \height) 
                            -- ++(-\width, 0) -- ++(-\depth*\scl, -\depth) -- cycle; 
                \draw[thick] (0, 0) -- ++(0, \height);    
                \draw[thick] (\width, 0) -- ++(0, \height); 
                \draw[thick] (0, \height) -- ++(\width, 0) -- ++(\depth*\scl, \depth); 
            \end{scope}
        }
    \end{scope}
    }
\end{tikzpicture}}
\caption{\label{fig:tensorastubes}A tensor as a matrix of tubes.}
\end{figure}

In this work, we flip the relation, and view the tubal precept as the underlying paradigm on which tubal products {\em emerges} in a bottom-up manner. In particular, we show that the $\Mprod$ product can be derived from the tubal precept is a series of logical steps, and that {\bf the end construction is unavoidable}. That is, the $\Mprod$ is not only \underline{an appropriate} way for construction tubal product, it is in fact \underline{the only way} to construct such a product that has all the necessary properties for a matrix-mimetic algebra. 



To elaborate, unlike previous works on tubal tensor algebra that start from the tensor-tensor product, and then show that it defines a commutative ring on tubes, we build things from the bottom-up: seek a commutative ring of tubes with desired properties, and reach the tubal product as a by-product. Even though the constructions and results are not new, we believe the bottom-up presentation is useful. 
We then proceed to showing that the construction in the first part is, in a sense, unavoidable if we accept the tubal precept. 
It is the only construction that has all the desired properties (commutativity, unitality, von Neumann regularity, and compatibility with the $\R$-vector space structure; see Section~\ref{sec:org4388465} for formal definitions). This is a new mathematical result, which we argue serves to demystify tubal tensor algebra, at least partially: $\Mprod$  is a ``right'' way to define a tensor product because it is the only way to define such a product under the framework of the tubal precept.



While partly expository in nature — aimed at presenting the foundations of tubal tensor algebra in a cohesive and accessible way —  this paper also addresses theoretical gaps in the tubal tensor framework, proves many new results, and provides justification for the tubal tensor framework central constructions,  thereby helping to demystify this important tensor framework. For completeness and to make the paper self contained as possible, we include proofs also for results that are not new.

\paragraph{Previous work on the foundations of tubal tensor algebra}
As mentioned earlier, the mathematical foundations of the tubal tensor framework was developed in a series of papers. We now expand on this. The t-product between two tensor was first defined in~\cite{Kilmer2008TR} (tech report) and~\cite{kilmer-2011-factor-strat} 
(published paper). In these papers, the t-product is defined in a purely algebraic way, as a process of folding the product of two matrices defined by the input tensors. 
In particular, the view of the t-product as a product between two matrices of tubes is not discussed in~\cite{Kilmer2008TR,kilmer-2011-factor-strat}. In addition, these papers prove several matrix-mimetic results on tensors, prove the existence of t-SVD, and discuss optimal approximations via t-SVD truncations.

Following initial work on the t-product, Braman made the connection between tensors and linear operators on the space of matrices, where linearity is restricted in the sense that it is with respect to scalars that are tubes~\cite{braman-2010-third-order}. To that end, this paper showed that the t-product defines a commutative ring on the space of tubes, and that the t-product between tensors is essentially a regular matrix-matrix product between two matrices of tubes.

Although implicitly present already in the t-SVD algorithm of~\cite{kilmer-2011-factor-strat}, the idea that various tubal tensor algorithms can be implemented by using the corresponding matrix algorithms applied face-wise in the transform domain was formalized in~\cite{kilmer-2013-third-order}. That paper also generalized many other concepts from matrix algebra to tubal tensor algebra, such as range, kernel, Gram-Schmidt, Krylov methods, etc. Power and Arnoldi iterations were also discussed by Gleich et al.~\cite{gleich-2012-power-arnol}.

Kernfeld et al. were the first to consider $\Mprod$ as a generalization of the t-product \cite{kernfeld-2015-tensor-tensor}. However, unlike previous works on the t-product, they assume tubes are over $\C$ not $\R$, likely to allow $\matM$ to be complex.
In contrast, we characterize exactly which complex $\matM$s might be used to define a commutative ring structure over real vector spaces.
Kernfeld et al. also show that $\KM$ (the commutative ring of tubes induced by $\matM$)\linebreak[3] forms a Hilbert C*-algebra over which $\KM^m$ is a Hilbert C*-module.

Eckart-Young-like optimality theorems for tensor-truncations based on the $\Mprod$ was proven in~\cite{kilmer-2021-tensor-tensor}, along with results that establish the superiority of such tensor-based compressions in comparison to matrix based compressions. 

Finally, we mention that the tubal tensor algebra has been generalized to tensors of orders higher than 3 in~\cite{MartinShaferLaRue13}.

\section{Preliminaries}
\label{sec:prelim}

\subsection{Notation}
We use lower case slanted bold letters for vectors ($\vx,\vy,\dots$),
upper case bold letters for matrices ($\matA,\matB,\dots$) and calligraphic
letters for higher order tensors ($\tenC,\tenD,\dots$).  We use $\hadprod$ to denote the Hadamard product (entry-wise product, i.e., $(\vx \hadprod \vy)_i \coloneqq x_i y_i$). In general, we use tensor terminology and notation from~\cite{kolda-2009-tensor-decom-applic}. In particular: 1) we use subscripts to index elements in a matrix or tensor, e.g. $\tenA_{ijk}$, and use MATLAB's $:$ notation to denote all entries of a mode. 2) Frontal slices of a tensor $\tenA$ are written as a matrix with the same letter and single index, e.g., $\matA_k \coloneqq \tenA_{::k}$. 3) Mode-$n$ tensor-matrix product is denoted using $\nmodeprod{n}$, e.g., $\tenA \nmodeprod{3} \matM$.

\subsection{Rings and Fields}
\label{sec:orgf5ad0cc}

To motivate tubal algebra we need some background from abstract algebra. We  only skim this huge topic,  introducing just the concepts that we need for our discussion, and stating motivating results without actually proving them. The following definitions are included for completeness and to fix notation; readers already familiar with abstract algebra may wish to skip ahead to Section~\ref{sec:org2fc628d} or Section~\ref{sec:orga63110b}. 

A \emph{binary composition} \(\cdot\) on a set \(S\) is a function  \(S \times S \ni (a,b) \mapsto a
\cdot b \in S\). 
A \emph{semigroup} is a pair \((S,\cdot)\) in which the binary composition \(\cdot\) is associative, i.e., for all \(a,b,c \in S\), \((a\cdot b)\cdot c = a\cdot (b\cdot c)\). 
A \emph{monoid} $(S,\cdot)$, is a semigroup with an identity element $e \in S$ such that \(e \cdot a = a \cdot e = a\) for all \(a \in S\). 
A \emph{group} $(S,\cdot)$ is a monoid  in which every element has an \emph{inverse}, i.e., for every \(a \in S\), there exists an element \(a^{-1} \in S\) such that \(a \cdot a^{-1} = a^{-1} \cdot a = e\). 
A group $(S, \cdot)$ is an \emph{abelian} or \emph{commutative} group if the  composition is a commutative operation, i.e., for all \(a,b \in S\), \(a \cdot b = b \cdot a\). 

A \emph{ring} is a set \((R,+, \cdot)\) with two binary compositions, \(+\) and \(\cdot\), such that \((R, +)\) is an abelian group, \((R, \cdot)\) is a semigroup, and multiplication is distributive over addition, i.e., for all \(a,b,c \in R\), \(a \cdot (b + c) = a \cdot b + a \cdot c\) and \((a + b) \cdot c = a \cdot c + b \cdot c\).

More verbosely, a ring is a set \(R\) with two binary compositions, \(+\) and \(\cdot\), such that the following \emph{ring axioms} hold:
\begin{itemize}
\item \((R, +)\) is an abelian group, meaning that:
\begin{itemize}
\item For all \(a,b,c \in R\), \((a + b) + c = a + (b + c)\) (associativity of addition).
\item For all \(a,b \in R\), \(a + b = b + a\) (commutativity of addition).
\item There exists an element \(0 \in R\) such that for all \(a \in R\), \(a + 0 = 0 + a = a\) (existence of an additive identity).
\item For every \(a \in R\), there exists an element \(-a \in R\) such that \(a + (-a) = (-a) + a = 0\) (existence of an additive inverse).
\end{itemize}
\item \((R, \cdot)\) is a semigroup, meaning that: 
\begin{itemize}
\item For all \(a,b,c \in R\), \((a \cdot b) \cdot c = a \cdot (b \cdot c)\) (associativity of multiplication).
\end{itemize}
\item Multiplication is distributive over addition, meaning that:
\begin{itemize}
\item For all \(a,b,c \in R\), \(a \cdot (b + c) = a \cdot b + a \cdot c\) and \((a + b) \cdot c = a \cdot c + b \cdot c\).
\end{itemize}
\end{itemize}

Rings can have additional structure (more axioms that they satisfy). A ring \(R\) is a \emph{ring with unity} (or {\em unital ring}) if \((R,\cdot)\) is a monoid, i.e., there exists an element \(1 \in R\) such that for all \(a \in R\), \(1 \cdot a = a \cdot 1 = a\). A ring \(R\) is a \emph{commutative ring} if the multiplication is commutative, i.e., for all \(a,b \in R\), \(a \cdot b = b \cdot a\). A nontrivial ring \(R\) is a \emph{division ring} if every non-zero element has a multiplicative inverse, i.e., for every \(a \in R\), \(a \neq 0\), there exists an element \(a^{-1} \in R\) such that \(a \cdot a^{-1} = a^{-1} \cdot a = 1\). Finally, a \emph{field} is a commutative division ring.

The notion of {\em \(*\)-ring} (read: ``star-ring'') requires the ring to also have a conjugation operation, which captures the notion of complex conjugation in the context of rings. A \(*\)-ring is a ring \(R\) equipped with a unary operation \(a \mapsto a^\sconj\) (called \emph{involution}) such that for all \(a,b \in R\), the following properties hold:
\begin{itemize}
\item \((a + b)^\sconj = a^\sconj + b^\sconj\)
\item \((a \cdot b)^\sconj = b^\sconj \cdot a^\sconj\)
\item \((a^*)^\sconj = a\)
\item If $R$ is a unital ring then \(1^\sconj = 1\)
\end{itemize}

\subsection{Modules and Algebras}
\label{sec:org2fc628d}


Note that a \emph{vector space} is essentially an abelian group endowed with a scalar multiplication defined over a field. Dropping the requirement that the scalars form a field and allowing them to come from a ring produces the broader construct known as a module.
A (left) \emph{module} over a ring \((R, +, \cdot)\), or \emph{\(R\)-module} for short, is an abelian group \((M, \oplus)\) together with a mapping  \(R \times M \ni (r, m) \mapsto r\odot m \in M\), such that
\begin{itemize}
\item \(r \odot (m \oplus m') = (r \odot m) \oplus (r' \odot m)\)
\item \((r + r') \odot m = (r \odot m) \oplus (r' \odot m)\)
\item \((r \cdot r') \odot m = r \odot (r' \odot m)\)
\item \(1 \odot m = m\) where \(1\) is the multiplicative unit over \(R\)
\end{itemize}
A result of these axioms, is that $0_R \odot m = r \odot 0_M = 0_M $ for all $r \in R,m \in M$ where $0_M, 0_R$ denote the additive identity elements of $M$ and $R$. 
It follows that $0_M = (1 - 1)\odot m = m \oplus (-1 \odot m)$, i.e., for any $m \in M$ we have that $(-1)\odot m$ is the additive inverse of $m$. 

The mapping $R \times M \to M$ is called \emph{scalar multiplication}. 
Consistent with the vector space conventions in linear algebra, we use the same symbol, $+$, to denote both the additive operation in the abelian group $M$ and that of the ring $R$. 
When context allows, we omit the $\odot$ notation and write $r \odot m = rm$ for $r \in R,m \in M$.

Vector spaces emerge as a special case of modules over rings in which the ring of scalars is a field. 
A module that is not a vector space behaves very similarly to a vector space, but with some differences. For example, not all modules have a basis, and among those that do, different bases may have different numbers of elements. A \emph{free module} is a module that has a basis. If the underlying ring is commutative (except the zero ring), then every basis of a free module has the same number of elements. This number is called the \emph{rank} of the module. If the rank is finite \(n\), the module is isomorphic to \(R^n\), and linear maps between such modules can be represented by matrices.

Let $R$ be a commutative ring. 
An \emph{associative $R$-algebra} is a ring $(A, + , \cdot)$ that is also an \(R\)-module $(A, +)$ such that 
$r(a\cdot a') = (r  a) \cdot a' =  a \cdot (r a')$ for all $a,a' \in A$ and $r \in R$, and the addition operations in $A$ with respect to its view both as a ring and as a module coincide
\[
(A,+,\cdot) \ni a + b = a + b \in (A, +)~.
\]
An $R$-algebra $(A, +, \cdot )$ is an {\em associative division algebra} if the ring $(A, +, \cdot )$ is a division ring. 

It is also possible to generalize the notion of inner product and norms, but we introduce these later when they become relevant. 

The main observation underlying tubal tensor algebra is that if we replace the field with a commutative ring with a unit (i.e., we drop only the requirement that every non-zero element has a multiplicative inverse), then matrices over that ring behave very much like matrices over a field, and many of the the powerful tools of linear algebra can come to bear in this setting\footnote{Actually, we shall see later that commutative ring with a unit is not enough, and we need another property of the ring for tools like SVD.}. 

\section{Algebra of Tubes}
In tubal tensor algebra we view a tensor \(\tenA \in \R^{m \times p \times n}\) as a matrix of \emph{tubes}, which are elements of \(\R^{1\times 1 \times n} \cong \R^n\). 

Ideally, we would have wanted the set of tubes to exhibit a field structure, since in this case $\tenA$ would have been a matrix of field elements, making all the powerful tools and theory of linear algebra applicable.

However, we have to forgo this due to the following theorem:
\begin{theorem}[Frobenius Theorem]
\label{thm:fro}
Every finite dimensional associative division algebra over the real numbers is isomorphic to \(\R\), \(\C\) or \(\mathbb{H}\) (the quaternions).
\end{theorem}
These algebras have real dimension 1, 2, and 4, respectively. Moreover, \(\mathbb{H}\) is not commutative, so it is not a field. So, if we insist on imposing a field structure on \(\R^n\) we will be working in dimension (\(n\)) at most 2, which is obviously highly restrictive.

Thus, the best we can hope to do is to use some ``supercharged'' ring that gets us as close as possible to a field. Based on the discussion in the previous section, if we define them as a commutative ring with a unit, then we are almost as good in the sense that modules are free modules, and these behave very much like vector spaces. Thus, in this section we investigate how to define such a ring of tubes. 

\subsection{Identifying Tubes with Polynomials}
\label{sec:orga63110b}

One way to define a ring of tubes is to use polynomials.  We can associate with a vector \(\va \in \R^n\) the degree \(n-1\)  polynomial \(P_\va \coloneqq a_1 + a_2 X + \dots + a_n X^{n-1} \in \R[X]\). We can define addition of tubes as addition of polynomials, and multiplication of tubes as multiplication of polynomials. For example, the tensor  \(\tenX\in\R^{3\times 4 \times 2}\) defined by frontal faces
\begin{equation}
\matX_1 = \left[\begin{array}{cccc}
                1 & 4 & 7 & 10 \\
                2 & 5 & 8 & 11 \\
                3 & 6 & 9 & 12
                \end{array}
          \right],
\quad
\matX_2 = \left[\begin{array}{cccc}
                13 & 16 & 19 & 22 \\
                14 & 17 & 20 & 23 \\
                15 & 18 & 21 & 24
                \end{array}
          \right] 
\end{equation}
might be identified with the matrix
\begin{equation}
\matX = \left[\begin{array}{cccc}
		1 + 13X & 4 + 16X & 7 + 19X & 10 + 22X \\
		2 + 14X & 5 + 17X & 8 + 20X & 11 + 23X \\
		3 + 15X & 6 + 18X & 9 + 21X & 12 + 24X
		\end{array}
	\right].
\end{equation}

\(\R[X]\) is a commutative ring with a unit, so it seems we can work with it. However, there is a problem. Multiplying two polynomials of non-zero degree will result in a polynomial of higher degree. This means that the product of two tensors will be a tensor in which the inward dimension is bigger than the dimension of the two multiplicands. In other words, the tubes are larger. This is not what we want.

To avoid this, we want to work with polynomials of only a fixed degree. We do not have an issue with addition, but we need to define how to multiply two polynomials of degree \(n-1\) that will give us a polynomial of degree \(n-1\) as well. 
One way to do so is to use modular multiplication of polynomials.

\subsection{Defining Tubal Product via Modular Multiplication of Polynomials}
\label{sec:orga100d4a}

Recall that two polynomials \(P\) and \(G\) are congruent modulo a third polynomial \(Q\), written as \(P \equiv G \mod Q\), if their difference is divisible by \(Q\). That is, \(P \equiv G \mod Q\) if and only if there exists a polynomial \(H\) such that \(P = G + QH\). For any polynomial \(P\) there is unique polynomial \(G\) of degree strictly less than the degree of \(Q\) 
such that \(P \equiv G \mod Q\). 
We denote the \emph{remainder} of \(P\) divided by \(Q\),  by \(P \mod Q\). 

We now endow the vector space of polynomials of degree less than \(n\) with the binary composition of modular multiplication modulo a polynomial \(Q\) of degree \(n\) as the multiplication, to obtain a ring, as long as all the roots of \(Q\) are distinct. 

Let's apply this idea to our setting of tubes. We already associated with a tube \(\va \in \R^n\) the polynomial \(P_\va\) of degree at most \(n-1\). We  define the product of two tubes \(\va\) and \(\vb\) as the tube associated with the polynomial \(P_\va \cdot P_\vb \mod Q\), where \(Q\) is some polynomial of degree \(p\). 

We are still left with specifying \(Q\). Let's consider the use of \(Q=X^n - 1\). Why this choice? The reason is that this leads to a very clean and elegant way to multiply tubes:  the Discrete Fourier Transform (DFT) can be used to define the result of the operation, and for efficiency the Fast Fourier Transform (FFT) can be used. A well known result on the connection between DFT and modular multiplication of polynomials (see, e.g., \cite{OppenheimSchaferBuck1999}, Section 8.6) implies that:
\begin{equation}
\label{eq:modmul-dft}
P_\vc = P_\va \cdot P_\vb \mod (X^p - 1) \quad \iff \quad \matF_n \vc = \matF_n \va \hadprod \matF_n \vb.
\end{equation}
where $\matF_n$ is the DFT matrix of size $n$.

Eq. (\ref{eq:modmul-dft}) implies that we can implement the idea that multiplication of two tubes as the tube associated with multiplying the polynomials module \(X^n - 1\) via 
\begin{equation}
\label{eq:tprod}
\va \tprod \vb \coloneqq \matF_n^{-1}(\matF_n \va \hadprod \matF_n \vb)
\end{equation}
By using FFT and inverse FFT to implement the DFT and its inverse, we can multiply two tubes in \(O(n \log n)\) operations. Note that for $\va,\vb\in\R^n$ then $\va \tprod \vb\in\R^n$.

Considering Eq.~\eqref{eq:tprod}, we immediately see that when endowing $\R^n$ with the usual vector addition and $\tprod$ product all the ring axioms hold, and we have a ring. Furthermore, we also see that the $\tprod$ is commutative, so we have a commutative ring.  Eq.~\eqref{eq:tprod} was originally proposed in~\cite{kilmer-2011-factor-strat}, and showing that $\R^n$ endowed with it leads to a commutative ring was originally shown in~\cite{braman-2010-third-order}.

\subsection{A Family of Tubal Products}
\label{sec:org9537a00}

Inspecting Eq. (\ref{eq:tprod}), we can come up with a family of tubal products. The idea is to replace \(\matF_n\) with an invertible matrix \(\matM\), and define the tubal product as
\begin{equation}
\label{eq:Mprod}
\va \Mprod \vb \coloneqq \matM^{-1}(\matM \va \hadprod \matM \vb)
\end{equation}
Eq. \eqref{eq:Mprod} was originally proposed in \cite{kernfeld-2015-tensor-tensor}.

We now discuss what requirements we should impose on $\matM$. Obviously, we want to allow complex matrices, so that $\matM=\matF_n$ is a valid choice. However, we cannot allow any complex matrix, since we want to the result of Eq. \eqref{eq:Mprod} to be real for any $\va,\vb\in \R^n$ (we are defining a product of two real vector that results in a real vector). The following lemma, which is new, characterizes such matrices. The proof is delegated to Section~\ref{sec:proof-winv-realness}.
\begin{lemma}
    Let  \(\matM\in\C^{n\times n}\) be an invertible matrix. 
    Then $\va \Mprod \vb \in \R^n$ for all $\va,\vb \in \R^n$ 
    , if and only if every row of $\matM$ is either real, or obtained by entry-wise complex conjugation of exactly one other row of $\matM$.
\end{lemma}

We now show that the $\Mprod$ product, together with the usual vector-space addition, endows $\R^n$ with a commutative ring structure, thus, form a ring of tubes. 
\begin{theorem}[Proposition 4.2 in \cite{kernfeld-2015-tensor-tensor}]
Let \(\matM\in\C^{n\times n}\) be an invertible matrix in which every row is either real, or is conjugate to exactly one other row of $\matM$. The set \(\R^n\) equipped with the binary composition \(\Mprod\) defined in Eq. \eqref{eq:Mprod} and the usual vector addition \(+\) is a commutative ring with a unit. We denote this ring by \(\KM\).
\end{theorem}

\begin{proof}
\(\KM\) is an abelian group with respect to the addition due to the usual properties of vector addition. We immediately see from the definition of \(\Mprod\) that it is commutative. We need to show that \(\Mprod\) is associative, distributive over addition, and has a unit.
\begin{itemize}

\item {\bf Associativity:} For all \(\va, \vb, \vc \in \KM\), we have
\begin{align*}
(\va \Mprod \vb) \Mprod \vc &= \matM^{-1}(\matM(\matM^{-1}(\matM \va \hadprod \matM \vb)) \hadprod \matM \vc) \\
&= \matM^{-1}((\matM \va \hadprod \matM \vb) \hadprod \matM \vc) \\
&= \matM^{-1}(\matM \va \hadprod (\matM \vb \hadprod \matM \vc)) \\
&= \va \Mprod (\vb \Mprod \vc)
\end{align*}
where the second equality uses \(\matM\matM^{-1} = \matI\), and the third uses associativity of the Hadamard product. 

\item {\bf Distributive over addition:} For all \(\va, \vb, \vc \in \KM\), we have \(\va \Mprod (\vb + \vc) = \matM^{-1}(\matM \va \hadprod \matM (\vb + \vc)) = \matM^{-1}(\matM \va \hadprod (\matM \vb + \matM \vc)) = \matM^{-1}(\matM \va \hadprod \matM \vb + \matM \va \hadprod \matM \vc) = \va \Mprod \vb + \va \Mprod \vc\).

\item {\bf Unit:} It is easy to verify that the unit element is \(\unitM \coloneqq \matM^{-1}\ve\) where $\ve$ is the $n$-dimensional vector of all ones. However, we still need to argue that it is real, so in the ring. This is an immediate corollary of Lemma~\ref{lem:struct-Minv} in Section~\ref{sec:proof-winv-realness}.
\end{itemize}
\end{proof}

\begin{example}[Complex numbers as a tubal ring]
\label{example:C-as-KM}
Consider the case \(n=2\). Since \(\R^2\) is isomorphic to \(\C\) we can use the multiplication operation defined by complex numbers to get a field:
\begin{equation*}
\begin{bmatrix} a \\ b \end{bmatrix} \cdot \begin{bmatrix} c \\ d \end{bmatrix} =
\begin{bmatrix} ac - bd \\ ad + bc \end{bmatrix}
\end{equation*}
This operation is actually \(\Mprod\) with 
\begin{equation*}
\matM = \begin{bmatrix} 1 & \si \\ 1 & -\si \end{bmatrix}
\end{equation*}
The use of complex numbers in \(\matM\) for building a $\Mprod$ which corresponds to complex product is unavoidable. It can be shown that no such real \(\matM\) exists.
\end{example}

\begin{proposition}
    The ring \(\KM\) is not a field, unless \(n=1\) or \(n=2\), since there are elements that do not have a multiplicative inverse (so it is not a division ring). 
\end{proposition}
\begin{proof}
    It is easy to verify that $\KM$ defines an associative algebra over $\R$, with the scalar product defined by the usual scalar product. Since $\KM = \R^n$ as a set and the algebra structure is compatible with the standard $\R$-vector space structure, we have $\dim_\R(\KM) = n$. Thus, if $\KM$ were a field, it would be a finite dimensional associative division algebra over $\R$, and by the Frobenius Theorem (Theorem~\ref{thm:fro}) isomorphic to $\R$, $\C$ or $\mathbb{H}$. However, it cannot be isomorphic to $\mathbb{H}$ since $\KM$ is commutative, while $\mathbb{H}$ is not. So, it must be isomorphic to $\R$ or $\C$. These have dimension, as a real vector space, of $1$ or $2$, so $n$ must be $1$ or $2$.
\end{proof}

We also define a partial order on \(\KM\). For \(\va, \vb \in \KM\), we write \(\va \Mleq \vb\) if \(\matM(\vb - \va)\) is real and element-wise non-negative. It is easy to verify that this is a partial order.

We already define \(\KM\) so that it is a commutative ring. 
To add additional structure, and make \(\KM\) a \(*\)-ring, we define the \emph{\(\Mprod\)-conjugate} of a tube \(\va\) as the tube 
\begin{equation}
\label{eq}
\va^\sconj \coloneqq \matM^{-1}\overline{\matM \va}
\end{equation}
where \(\overline{\cdot}\) denotes the complex conjugate elementwise. For this definition to make sense, we need that $\va^\sconj\in\R^n$ for every $\va\in\R^n$. This holds if $\matM^{-1} \overline{\matM}$ is a real matrix. We show this in Section~\ref{sec:proof-winv-realness}. We now show that endowing \(\KM\) with the \(\Mprod\)-conjugate operation results in a \(\star\)-ring. 

\begin{proposition}
\label{prop:KM-star-ring}
\(\KM\) is a \(*\)-ring with respect to the \(\Mprod\)-conjugate operation.
\end{proposition}
\begin{proof}
We verify the four \(*\)-ring axioms:
\begin{itemize}
\item {\bf Involutive:} \((\va^\sconj)^\sconj = \matM^{-1}\overline{\matM(\matM^{-1}\overline{\matM\va})} = \matM^{-1}\overline{\overline{\matM\va}} = \matM^{-1}\matM\va = \va\).
\item {\bf Additive:} \((\va + \vb)^\sconj = \matM^{-1}\overline{\matM(\va+\vb)} = \matM^{-1}(\overline{\matM\va} + \overline{\matM\vb}) = \va^\sconj + \vb^\sconj\).
\item {\bf Anti-multiplicative:} \((\va \Mprod \vb)^\sconj = \matM^{-1}\overline{\matM\va \hadprod \matM\vb} = \matM^{-1}(\overline{\matM\va} \hadprod \overline{\matM\vb}) = \va^\sconj \Mprod \vb^\sconj\). Since \(\KM\) is commutative, this equals \(\vb^\sconj \Mprod \va^\sconj\).
\item {\bf Unit-preserving:} \(\unitM^\sconj = \matM^{-1}\overline{\matM\unitM} = \matM^{-1}\overline{\mathbf{1}} = \matM^{-1}\mathbf{1} = \unitM\).
\end{itemize}
\end{proof}

\subsection{Weak Inverses}
\label{sec:orgb12cbf2}

\(\KM\) is almost a field: It has a unit and it is commutative, yet, it does not contain multiplicative inverse for all non-zero element. 
That is, \(\KM\) is not a division ring. 
However, we shall now see that every element in \(\KM\) has a \emph{weak inverse}.
 
For a vector \(\va\), let us denote by \(\va^\pinv\) the vector obtained by inverting non-zero entries, and keeping zero entries zero. That is,
\begin{equation*}
(\va^\pinv)_i = \begin{cases} a^{-1}_i & \text{if } a_i \neq 0 \\ 0 & \text{if } a_i = 0 \end{cases}
\end{equation*}
Next, define 
\begin{equation}
\label{eq:def-winv}
\va^\winv \coloneqq \matM^{-1}(\matM \va)^\pinv
\end{equation}
\begin{lemma}
\label{lem:winv-realness}
Let $\matM\in\C^{n\times n}$ be an invertible matrix. Suppose that every row of $\matM$ is either real, or is conjugate to exactly one other row of $\matM$. Then $\vx^{\winv}$ is real, for every $\vx\in\R^n$.
\end{lemma}

The proof is delegated to Section~\ref{sec:proof-winv-realness}. We now show that \(\va^\winv\) acts like a weak inverse. In fact, it acts like the Moore-Penrose pseudoinverse of matrices.
\begin{lemma}
For all \(\va \in \KM\), \(\va^\winv\in\KM\) is the unique element of \(\KM\) for which the following four properties hold:
\begin{itemize}
\item \(\va \Mprod \va^\winv \Mprod \va = \va\)
\item \(\va^\winv \Mprod \va \Mprod \va^\winv = \va^\winv\)
\item \((\va \Mprod \va^\winv)^\sconj = \va \Mprod \va^\winv\)
\item \((\va^\winv \Mprod \va)^\sconj = \va^\winv \Mprod \va\)
\end{itemize}
\end{lemma}

\begin{proof}
The first two properties follow immediately from the fact that for every $\vx$, we have \(\vx \hadprod \vx^\pinv \hadprod \vx = \vx\) and \(\vx^\pinv \hadprod \vx \hadprod \vx^\pinv = \vx^\pinv\).  The first property makes the ring a von Neumann regular ring (see Definition~\ref{def:vnr}), and for commutative von Neumann rings we have a unique element for which both properties hold~\cite{LombardiQuitte2015}.  As for the last two properties, note that \(\matM(\va \Mprod \va^\winv) = \matM\va \hadprod (\matM\va)^\pinv\), whose entries are either \(0\) or \(1\), and hence real. Since real elements are self-conjugate (if \(\matM\vb\) is real then \(\vb^\sconj = \matM^{-1}\overline{\matM\vb} = \matM^{-1}\matM\vb = \vb\)), the conjugation conditions \((\va \Mprod \va^\winv)^\sconj = \va \Mprod \va^\winv\) and \((\va^\winv \Mprod \va)^\sconj = \va^\winv \Mprod \va\) are automatically satisfied. Uniqueness of the element satisfying all four conditions follows from the uniqueness under the first two conditions alone, since the last two do not impose additional constraints.
\end{proof}

For an element \(\va \in \KM\), the element \(\va^\winv\) acts like a weak inverse in the following sense. If \(\va\) has a multiplicative inverse, then \(\va^\winv\) is that inverse. If \(\va\) does not have a multiplicative inverse, then \(\va^\winv\) is the closest thing to it, in the sense that it acts in many ways like a multiplicative inverse. The fact that \(\KM\) has weak inverses makes the it an even-closer-to-a-field ring.

\begin{definition}
\label{def:vnr}
A ring \(R\) is called a \emph{von Neumann regular ring} if for every element \(a \in R\) there exists an element \(x \in R\) such that \(a = a \cdot x \cdot a\).
\end{definition}

We remark that it is known that for a commutative von Neumann regular ring, for every element \(a\in\R\) there exists a \emph{unique} element \(x\in R\) such that \(a = a\cdot x\cdot a\) and \(x = x\cdot a \cdot x\). In the above we have shown that additional properties hold as well.

\subsection{\texorpdfstring{\(\Mprod\)}{M-product} is Unavoidable}
\label{sec:org4388465}

Although we reached the definition of \(\KM\) in some logical progression of ideas, it is plausible that another definition could have been reached. In fact, even \(\Mprod\) itself might be derived via other reasoning (see next subsection). So, one might ask why this form of tubal product, and not some other one? We now claim that, in a sense, \(\Mprod\) is unavoidable, by stating the paper's main result: every ring on $\R^n$ that has all the desired properties must be \(\KM\) for some \(\matM\).

\begin{tcolorbox}[colback=red!5!white,colframe=red!75!black]
\begin{definition}
A ring \((\R^n, +, \cdot)\), where \(+\) is the usual vector addition and \(\cdot\) is an arbitrary binary operation for which the ring axioms hold, is called a \emph{tubal ring} if it is commutative, unital, von Neumann regular, and is also an associative algebra over \(\R\) with respect to the usual scalar-vector product.
\end{definition}

\begin{theorem}[Main Result]
\label{thm:main}
Suppose that \((\R^n, +, \cdot)\) is a tubal ring. Then, there exists an invertible matrix \(\matM\in\C^{n\times n}\) such that for all \(\va, \vb \in \R^n\), we have \(\va \cdot \vb = \va \Mprod \vb\).
\end{theorem}
\end{tcolorbox}
In order to streamline the presentation, we defer the proof to Section~\ref{sec:proof}.  As Example \ref{example:C-as-KM} shows, the last theorem does not hold if we restrict \(\matM\) to be real.

The requirement of being a von Neumann regular ring, i.e., the existence of a weak inverse for every element, might seem technical and possibly redundant. 
In fact, as we shall see in shortly in Example~\ref{example:dual-numbers}, it possible to define a commutative ring over $\R^2$ (for example) that is not von Neumann regular.

First, we note that the requirement of having a weak inverse brings the tubal ring very close, possibly as closest as possible, to being a field. Since we cannot define a field over $\R^n$ (for $n>2$), this is desirable. Secondly, in Appendix~\ref{app:additional-discussion} we show that a commutative ring with elements in $\R^n$ which is not von Neumann regular must have nilpotent elements, i.e., elements $\vn \in\R^n$ such that $\vn^2 = 0$. Obviously, this is a somewhat pathological situation for elements that we want to serve as scalars. We conjecture that the existence of such elements precludes a meaningful definition for unitary tensors, with such definition essential for defining a tensor SVD. We leave this for future research.   

The fact that for a tubal ring there exist such a $\matM$ implies that, in a sense, there is a small set of possible tubal rings. 
\begin{definition}[Realness]
    Consider the tubal ring $\KM$ for an $n$-by-$n$ matrix $\matM$, and recall that each row of $\matM$ is either real, or has a unique conjugate pair. 
    The {\em realness} of $\KM$ or $\matM$ is the number of real rows in $\matM$. 
    The ring $\KM$ is said to be {\em fully real} or {\em real-like} if its realness is equal to $n$.
\end{definition}
\begin{proposition}
    \label{prop:realness-iso}
    Let $\matM$ and $\matM'$ be of equal size. The rings $\KM$ and ${\mathbb{K}}_{\matM'}$ are isomorphic as rings if their realness is the same. 
\end{proposition}
The proof is delegated to Section~\ref{sec:proof-winv-realness}. 
The name ``real-like'' is justified since for real-like tubal rings we have $\va^\sconj=\va$ for all $\va$. 
Since complex rows in $\matM$ come in conjugate pairs, the realness $r$ must satisfy $r \equiv n \pmod{2}$, i.e., $r \in \{0, 2, 4, \dots, n\}$ if $n$ is even, or $r \in \{1, 3, 5, \dots, n\}$ if $n$ is odd. Thus there are, up to isomorphism, $\lfloor \nicefrac{n}{2} \rfloor + 1$ tubal rings. However, this statement might be a bit misleading, in the sense that the isomorphism might not be a isometry (assuming that we have imposed a norm on tubes). 

\subsubsection{Examples}
\label{subsec:examples}

\begin{example}[Split-complex numbers]
    Up to isomorphism, every commutative ring over $\R^2$ is isomorphic to one of three possible rings: complex numbers, split-complex numbers and dual numbers~\cite{jaglom-1968-complex}. We already seen that the first is a tubal ring. We now consider the second, and discuss the third in the next example.

    Split-complex numbers are are expressions of the form $a + b\sj$, where $a,b\in\R$ and $\sj$ is a symbol taken to satisfy $\sj^2 = 1$. Addition and product are naturally defined using this rule:
   \begin{align*}
      (a + b\sj) + (c + d\sj) &= (a + c) + (b + d)\sj \\
      (a + b\sj) \cdot (c + d\sj) &= (ac + bd) + (ad + bc)\sj
   \end{align*}
   To express split-complex number product as a product over $\R^2$ we can use:
   \begin{equation*}
        \begin{bmatrix} a \\ b \end{bmatrix} \cdot \begin{bmatrix} c \\ d \end{bmatrix} =
        \begin{bmatrix} ac + bd \\ ad + bc \end{bmatrix}
    \end{equation*}
    This clearly defines a commutative ring over $\R^2$. Furthermore, it can be verified that weak inverses for non-zero elements are given by:
    \begin{align*}
        \begin{bmatrix} a \\ b \end{bmatrix}^\winv &= \begin{bmatrix} a/(a^2-b^2) \\ -b/(a^2-b^2)\end{bmatrix}\,\,(\text{for }a^2\neq b^2), \\
        \begin{bmatrix} a \\ a \end{bmatrix}^\winv &=
        \begin{bmatrix} 1/(4a) \\ 1/(4a) \end{bmatrix}, \quad\quad \begin{bmatrix} a \\ -a \end{bmatrix}^\winv =
        \begin{bmatrix} 1/(4a) \\ -1/(4a) \end{bmatrix}
    \end{align*}
    So, this a tubal ring. Indeed, product is a $\Mprod$ with\footnote{This matrix was found using the techniques reported in Appendix~\ref{app:additional-discussion}.}
    $$
        \matM = \begin{bmatrix}  1 & 1 \\ 1 & -1  \end{bmatrix}
    $$
    which is the $2 \times 2$ Hadamard matrix (which also coincides with the unscaled DFT matrix for $n=2$). We see that the split-complex product is the t-product for $n=2$.
    
\end{example}

To better understand what tubal rings are so that we are able to recognize such rings when we see them, it is useful to also observe an example to what tubal rings are not. 
\begin{example}[Dual numbers]
\label{example:dual-numbers}
   Introduced by Clifford in 1873~\cite{clifford-1873-biquaternions}, with applications in linear algebra and physics, dual numbers is a number system defined by two reals. 
   
   Dual numbers are expressions of the form $a + b\varepsilon$, where $a,b\in\R$ and $\varepsilon$ is a symbol taken to satisfy $\varepsilon^2 = 0$ and $\varepsilon\neq 0$. Addition and product are naturally defined using this rule:
   \begin{align*}
      (a + b\varepsilon) + (c + d\varepsilon) &= (a + c) + (b + d)\varepsilon \\
      (a + b\varepsilon) \cdot (c + d\varepsilon) &= ac + (ad + bc)\varepsilon
   \end{align*}
   To express dual number product as a product over $\R^2$ we can use:
   \begin{equation*}
        \begin{bmatrix} a \\ b \end{bmatrix} \cdot \begin{bmatrix} c \\ d \end{bmatrix} =
        \begin{bmatrix} ac\\ ad + bc \end{bmatrix}
    \end{equation*}
    This defines a commutative ring over $\R^2$, however it is {\em not} a tubal ring. The reason is that not every element has a weak inverse. To see this, note that any dual number $x+y\varepsilon$:
    \begin{equation*}
\varepsilon\cdot(x+y\varepsilon)\cdot\varepsilon = 0
    \end{equation*}
    so $\varepsilon$ does not have a weak inverse. Indeed, no invertible $\matM$ exists that makes dual numbers isomorphic to $\R^2$. 
\end{example}

\begin{example}[t-product via circulant matrices~\cite{gleich-2012-power-arnol}]
    Given $\vx=[x_1,\dots,x_n]$ let $\circmatrix(\vx)$ denote the circulant matrix whose first column is $\vx$, i.e., 
    \begin{equation*}
        \circmatrix(\vx)\coloneqq 
        \begin{bmatrix}
            x_1     & x_n     & \cdots & x_3    & x_2 \\
            x_2     & x_1     & x_n    &        & x_3 \\
            \vdots  & x_2     & x_1    & \ddots & \vdots \\
            x_{n-1} &         & \ddots & \ddots & x_n \\
            x_n     & x_{n-1} & \cdots & x_2    & x_1
        \end{bmatrix}
    \end{equation*}
    The sum of circulant matrices is circulant, the product of circulant matrices is also circulant. Any two circulant matrices commute. Finally, the identity matrix is a circulant matrix. Thus, by identifying $\vx$ with $\circmatrix(\vx)$ we define a commutative unital ring. In particular, $\circmatrix(\vx \tprod \vy) = \circmatrix(\vx) \cdot \circmatrix(\vy)$, where the right-hand side is standard matrix multiplication. This ring is tubal, with $\matM$ equal to the DFT matrix, i.e., this ring is exactly the t-product. In fact, the original construction of the t-product in \cite{kilmer-2011-factor-strat, braman-2010-third-order} was based on circulant matrices (see also \cite{gleich-2012-power-arnol}).
\end{example}

\begin{example}[Negacyclic convolution]
The t-product is the result of using polynomial product modulo $X^n - 1$. However, we can choose some other polynomial $Q$, as long as its degree is $n$ (for a ring over $\R^n$) and all its roots are distinct. We have
\begin{equation*}
P_\vc = P_\va \cdot P_\vb \mod Q \quad \iff \quad \matM \vc = \matM \va \hadprod \matM \vb.
\end{equation*}
where $\matM$ is the Vandermonde matrix at the roots of $Q$, so polynomial product modulo $Q$ is essentially $\Mprod$ with this $\matM$.

A particularly interesting example is $Q=X^n + 1$. Here, $\matM$ is the {\em skew-Fourier} matrix:
\begin{equation*}
    \matM = \matF_n 
    \begin{bmatrix}
        1 &               &               &        &     \\
          & e^{\pi \si /n}  &               &        &     \\
          &               & e^{2\pi \si /n} &        &     \\
          &               &               & \ddots &     \\
          &               &               &        & e^{(n-1)\pi \si /n}
    \end{bmatrix}
\end{equation*}
When $n=2$, $\matM$ is the same matrix as in Example~\ref{example:C-as-KM}, so the ring we obtain is isomorphic to $\C$ and is a field. In fact, in textbooks showing that $\C$ is isomorphic as a field to $\R[x]/(x^2+1)$ is sometimes given as an exercise~\citep[Section 9.4, Exercise 7]{dummit-2009-abstractalgebra}. So, in a sense, this is a more natural choice than $X^n -1$ (which leads to the t-product) for constructing a tubal ring.   
\end{example}

\begin{example}[Ring group~\cite{navasca2010tensors}]
    The following is a general technique for constructing a tubal rings. Previous examples are instances of it. Let $G$ be a finite abelian group of size $n$. The {\em group ring} $\R[G]$ are the set of functions $G\to\mathbb{R}$, with point-wise addition, and convolution as product:
    \begin{equation*}
        (f\cdot g)(x)= \sum_{u\cdot v=x} f(u) g(v)
    \end{equation*}
    This defines a commutative von Neumann regular ring. To make it a tubal ring, simply identify each group element $g \in G$ with an index in $\{1,\dots, n\}$.

    One example of this process is the t-product, which is obtained by using the cyclic group $G=\Z_n$. Another example is as follows. Consider the group $\Z^k_2=\{0,1\}^k$ with addition modulo 2 (componentwise XOR). It has $n=2^k$ elements. Index each group element by an integer in $0,\dots,n-1$, corresponding to its binary representation. If we index vectors in $\R^n$ using $0,\dots,n-1$ as well, convolution of functions over  $\Z^k_2$ defines the following product:
    \begin{equation*}
        (\va \cdot \vb)_i = \sum^{n-1}_{j=0} \va_i \vb_{i\oplus j}
    \end{equation*}
    where $\oplus$ denotes the bitwise XOR operation. 
    Since over the group $\Z^k_2$\linebreak XOR-convolution diagonalizes under the Walsh–Hadamard transform (WHT), exactly like ordinary cyclic convolution diagonalizes under the DFT, this is a tubal ring with $\matM=\mat{WHT}$.

\end{example}
\section{Module of Oriented Matrices}
\label{sec:org6426c78}

We define \(\KM^m\) to be the direct sum of \(m\) copies of \(\KM\), that is \(\KM^m \coloneqq \KM \oplus \cdots \oplus \KM\). We view this as a column vector of \(m\) elements of \(\KM\). Addition is element addition, and multiplication by a tube is element-wise multiplication. 

\begin{figure}[htbp]
\centering
{\usetikzlibrary{decorations.pathreplacing} 

\begin{tikzpicture}
    \def\width{0.35}    
    \def\height{\width}   
    \def\depth{2.3*\width}   
    \def\scl{0.8} 
    \def\spacing{.45}     

    \begin{scope}
        \foreach \x in {0} {
        \begin{scope}[shift={(\x*\spacing,0)}]
            \foreach \y in {0,1,2,3} {
                \begin{scope}[shift={(0, \y * \spacing)}]
                    \fill[blue!60] (0, 0) -- ++(\width, 0) -- ++(0, \height) -- ++(-\width, 0) -- cycle;
                    \fill[blue!40] (0, \height) -- ++(\width, 0) -- ++(\depth*\scl, \depth) -- ++(-\width, 0) -- cycle;
                    \fill[blue!20] (\width, 0) -- ++(0, \height) -- ++(\depth*\scl, \depth) -- ++(0,-\height) -- cycle;
                    \draw[thick] (0, 0) -- ++(\width, 0) -- ++(\depth*\scl, \depth) -- ++(0, \height) 
                                -- ++(-\width, 0) -- ++(-\depth*\scl, -\depth) -- cycle;
                    \draw[thick] (0, 0) -- ++(0, \height);
                    \draw[thick] (\width, 0) -- ++(0, \height);
                    \draw[thick] (0, \height) -- ++(\width, 0) -- ++(\depth*\scl, \depth);
                \end{scope}
            }
        \end{scope} 
        }

        \draw [decorate, decoration={brace,  amplitude=5pt}, thick] 
            (-0.2, 0) -- (-0.2, \spacing*4-0.1) node[midway,xshift=-0.4cm] {\footnotesize $m$};
        \draw [decorate, decoration={brace, amplitude=5pt}, thick] 
            (-0.2, \spacing*4) -- ++(\depth*\scl, \depth) node[midway,left,xshift=-0.1cm, yshift=0.3cm] {\footnotesize $\KM$};
    \end{scope}

    \begin{scope}[xshift=6.5cm]
        \pgfmathsetmacro{\len}{sqrt(\depth*\depth * 6)}
        \pgfmathsetmacro{\totheight}{\spacing * 4}
        \def\matrixdepth{0.15}

        \fill[blue!20] (0, 0) -- ++(\len, 0) -- ++(0, \totheight) -- ++(-\len, 0) -- cycle;
        \fill[blue!40] (0, \totheight) -- ++(\len, 0) -- ++(\matrixdepth, \matrixdepth) -- ++(-\len, 0) -- cycle;
        \fill[blue!60] (\len, 0) -- ++(0, \totheight) -- ++(\matrixdepth, \matrixdepth) -- ++(0, -\totheight) -- cycle;

        \draw[thick] (0, 0) -- ++(\len, 0) -- ++(0, \totheight) -- ++(-\len, 0) -- cycle;
        \draw[thick] (0, \totheight) -- ++(\len, 0) -- ++(\matrixdepth, \matrixdepth) -- ++(-\len, 0) -- cycle;
        \draw[thick] (\len, \totheight) -- ++(0, -\totheight);
        \draw[thick] (\len, 0) -- ++(\matrixdepth, \matrixdepth) -- ++(0, \totheight);
        \draw[thick] (\len + \matrixdepth, \totheight + \matrixdepth) -- ++(-\len, 0);

        \draw [decorate, decoration={brace,  amplitude=5pt}, thick](-0.2, 0) -- ++(0, \totheight) node[midway,xshift=-0.4cm] {$m$};
        \draw [decorate, decoration={brace, amplitude=5pt}, thick] (\matrixdepth, \totheight+0.3) -- ++(\len, 0) node[midway,yshift=0.4cm] {$n$};
    \end{scope}

    \draw[->, thick] (2,1.5) -- (5.5,1.5) node[midway,above] {\textit{squeeze}};
    \draw[<-, thick] (2,1.0) -- (5.5,1.0) node[midway,below] {\textit{twist}};
\end{tikzpicture}}
\caption{\label{fig:twist-and-squeeze}Oriented matrices: isomorphism between \(\KM^m\) and \(\R^{m\times n}\). Going from right to left (``twist''), each row of the matrix is lifted into the third dimension to become a tube element; going from left to right (``squeeze''), the tube entries are laid flat back into matrix rows.}
\end{figure}

\(\KM^m\) is isomorphic to \(\R^{m \times n}\): it can be viewed as taking the rows of such matrices and ``twisting'' them into the third dimension. The inverse operation is ``squeezing''. See Figure \ref{fig:twist-and-squeeze} for visual illustration. Thus, we refer to the space \(\KM^m\) as the space of \emph{oriented matrices}, and denote them (typesetting wise) as matrices. 

\begin{theorem}[Theorem 4.1 in~\cite{kernfeld-2015-tensor-tensor}]
\(\KM^m\) is a free module over \(\KM\) of rank \(m\).
\end{theorem}

\begin{proof}
Since we are using a direct sum of \(m\) modules, \(\KM^m\) is immediately a free module over \(\KM\) (every finite cartesian product of rings is a a free module). Nevertheless, it is instructive to give an explicit basis for \(\KM^m\). A natural basis for \(\KM^m\) is the set of \(m\) oriented matrices \(\matE_1, \dots, \matE_m\) where \(\matE_i\) is the oriented matrix with all zeros except for the \(i\)-th tube, which is the unit element of \(\KM\). It is easy to verify that this is a basis. 

We have shown that \(\KM^m\) has a basis, so it is a free module. The statement that it has rank \(m\) is short for saying that every basis has \(m\) elements. Since \(\KM\) is a commutative ring, it has {\em invariant basis number}, which means that all bases of a finitely generated free module over \(\KM\) have the same number of elements. Since we demonstrated a basis with \(m\) elements, all bases must have \(m\) elements, and the rank is \(m\).
\end{proof}

We go further and consider inner products. Given two oriented matrices \(\matU, \matV \in \KM^m\), we define their \emph{\(\Mprod\)-dot product} as
\begin{equation}
\label{eq:starM-dotprod}
\matU \Mdot \matV \coloneqq \sum_{i=1}^m \vu^\sconj_i \Mprod \vv_i.
\end{equation}
Notice that this is a sum of \(m\) \(\Mprod\)-products, and thus it is an element of \(\KM\), not a scalar. In tubal algebra scalars are tubes. The following theorem shows that the \(\Mdot\) behaves like an inner product. 
\begin{theorem}
For all \(\matU, \matV, \matW \in \KM^m\) and \(\vr,\vs \in \KM\), the \(\Mdot\) product satisfies the following properties:
\begin{itemize}
\item {\bf \(\Mprod\)-conjugate symmetry:} \(\matU \Mdot \matV = (\matV \Mdot \matU)^\sconj\).
\item {\bf \(\Mprod\)-linearity in the second argument:} \(\matU \Mdot (\vr \Mprod \matV + \vs \Mprod \matW) = \vr \Mprod (\matU \Mdot \matV) + \vs \Mprod (\matU \Mdot \matW)\).
\item {\bf \(\Mprod\)-positive definiteness:} \(\matU \Mdot \matU \Mgeq 0\) and \(\matU \Mdot \matU = 0\) if and only if \(\matU = 0\).
\end{itemize}
\end{theorem} 
\begin{proof}
The only non-trivial property is the last one. We will show that for all \(\vu \in \KM\) we have \(\vu^\sconj \Mprod \vu \Mgeq 0\), and \(\vu^\sconj \Mprod \vu = 0\) if and only if \(\vu = 0\). The claim then follows immediately from the definition of the \(\Mprod\)-dot product. To show that \(\vu^\sconj \Mprod \vu \Mgeq 0\) we need to show \(\R \ni \matM(\vu^\sconj \Mprod \vu) \geq 0\). Let \(\vy=\matM\vu\). By the definition of \(\Mprod\) we have
\begin{align*}
\matM(\vu^\sconj \Mprod \vu) &= \matM(\matM^{-1}(\matM \vu^\sconj \hadprod \matM \vu)) \\
&= \matM \vu^\sconj \hadprod \matM \vu \\
&= \matM (\matM^{-1} \overline{\matM \vu}) \hadprod \matM \vu \\
&= \overline{\matM \vu} \hadprod \matM \vu \\
&= \overline{\vy} \hadprod \vy 
\end{align*}
Element \(j\) of \(\overline{\vy} \hadprod \vy\) is equal to \(|y_j|^2\) which is real and non-negative, and zero only if \(y_j=0\). Since \(\matM\) is invertible, \(\vy = \matM\vu = 0\) if and only if \(\vu = 0\), so \(\vu^\sconj \Mprod \vu = 0\) if and only if \(\vu = 0\). It follows that \(\matU \Mdot \matU = \sum_{i=1}^m \vu_i^\sconj \Mprod \vu_i \Mgeq 0\), and  \(\matU \Mdot \matU = 0\) if and only if each summand is zero, which happens if and only if \(\vu_i = 0\) for all \(i\), i.e., \(\matU = 0\).
\end{proof}
\section{Tensors as Matrices over \texorpdfstring{\(\KM\)}{KM}}
\label{sec:org1bd601a}

We can identify a tensor \(\tenA \in \R^{m \times p \times n}\) with a matrix of elements of \(\KM\), that is \(\tenA \in \KM^{m \times p}\). We define the \(\Mprod\)-product of two tensors \(\tenA\in\KM^{m\times p}\) and \(\tenB\in\KM^{p\times l}\) as the tensor obtained using usual matrix multiplication, but with the \(\Mprod\)-product instead of the usual product. That is, 
\begin{equation*}
(\tenA \Mprod \tenB)_{ij} = \sum_{k=1}^p \tenA_{ik} \Mprod \tenB_{kj}
\end{equation*}
We then, again, identify the result with a tensor in \(\R^{m \times l \times n}\).

In linear algebra, matrices over fields are closely connected to linear maps between vector spaces. In tubal algebra, a similar property holds for third-order tensors, on the module of oriented matrices. 
\begin{definition}
A function \(T:\KM^p \to \KM^m\) is \(\Mprod\)-linear if for all \(\matU, \matV \in \KM^p\) and \(\vr,\vs \in \KM\), we have 
\begin{equation*}
T(\vr \Mprod \matU + \vs \Mprod \matV) = \vr \Mprod T(\matU) + \vs \Mprod T(\matV).
\end{equation*}
A function \(S:\R^{p\times n} \to \R^{m\times n}\) is \(\Mprod\)-linear if \((\twist \circ S\circ \squeeze)\) is \(\Mprod\)-linear.
\end{definition}
\begin{remark}
  Since for any vector \(\vu\in\R^n\) and scalar \(\alpha \in \R\) we have \(\alpha \vu = (\alpha\unitM)\Mprod \vu\), every \(\Mprod\)-linear function is also linear. However, the converse is not true.  
\end{remark}

Let \(\tenA\in\KM^{m\times p}\) be a tensor, and consider the map \(T:\R^{p\times n} \to \R^{m\times n}\) defined by \(T(\matX) = \squeeze(\tenA \Mprod \twist(\matX))\), where we wrote the twisting and squeezing operation explicitly (instead of relying on the isomorphism between matrices and oriented matrices). We easily see that \(T\) is \(\Mprod\)-linear. Furthermore, every \(\Mprod\)-linear function can be represented in this way. In addition, composition of \(\Mprod\)-linear functions is \(\Mprod\)-linear, and the representing tensor is the \(\Mprod\)-product of the representing tensors of the functions. Thus, in the tubal algebra we have tensors as linear operators on matrices, albeit a restricted type of linear operators (\(\Mprod\)-linear operators)~\cite{braman-2010-third-order}. 

\subsection{*-Algebra Structure}
 We define two additional operations to make tensors of the form $\KM^{m\times m}$ a \(*\)-algebra. First, scalar multiplication (where the scalar is a tube in \(\KM\)) is defined as element-wise multiplication. Second, we define the \(\Mprod\)-Hermitian transpose of a tensor \(\tenA\) via
\begin{equation*}
(\tenA^\ha)_{ij} \coloneqq \tenA_{ji}^\sconj
\end{equation*}
We also extend the definition to non-square tensors, i.e., $\KM^{m\times p}$ with $m\neq p$, using the same formula.

Algebraically, the $\Mprod$-Hermitian transpose behaves like the conjugate transpose on matrices: for every compatibly sized $\tenA$ and $\tenB$ we have $(\tenA \Mprod \tenB)^\ha = \tenB^\ha \Mprod \tenA^\ha$~\cite{kernfeld-2015-tensor-tensor}. More importantly, the  $\Mprod$-Hermitian transpose behaves like we expect in relation to adjoints of $\Mprod$-linear maps. 

Recall that in an inner product space, the adjoint of a linear map $T$ is the unique linear map $T^\adj$ such that for all $\vx$ and $\vy$ we have $\inner{T(\vx)}{\vy} = \inner{\vx}{T^\adj(\vy)}$. Furthermore, if the inner product is the usual dot product, and $T(\vx)=\matA\vx$ for some matrix $\matA$, then $T^\adj(\vy)=\matA^\ha \vy$. A similar property hold the $\Mprod$-Hermitian transpose with the \(\Mprod\)-dot product replacing the usual dot product.
\begin{definition}
    The $\Mprod$-adjoint of a $\Mprod$-linear operator \(T:\KM^p \to \KM^m\) with respect to \(\Mprod\)-dot product is the unique operator  \(T^\adj:\KM^m \to \KM^p\) such that for all $\matX\in\KM^m$ and $\matY\in\KM^p$ we have $T(\matX)\Mdot\matY=\matX\Mdot T^\adj(\matY)$.
\end{definition}
\begin{proposition}
    Suppose that $T(\matX)=\tenA \Mprod \matX$. Then, $T^\adj(\matY)=\tenA^{\ha} \Mprod\matY$.
\end{proposition}
\begin{proof}
    Eq.~\eqref{eq:starM-dotprod} reads as $\matU\Mdot \matV = \matU^\ha \Mprod \matV$, where $\matU^\ha$ denotes the $1 \times m$ row matrix (over $\KM$) $\left[\vu_1^\sconj, \ldots, \vu_m^\sconj\right]$ and $\matU^\ha \Mprod \matV$ is the $1\times 1$ matrix product over $\KM$, naturally identified with an element of $\KM$. So, for every $\matX$ and $\matY$,
    \[
    T(\matX)\Mdot \matY = (\tenA \Mprod \matX)^\ha \Mprod \matY = \matX^\ha \Mprod \tenA^\ha \Mprod \matY = \matX \Mdot (\tenA^\ha \Mprod \matY)
    \]
    so $T^\adj(\matY)=\tenA^\ha \Mprod\matY$ from the definition of $T^\adj$.
\end{proof}
Note that the claim that $\Mprod$-Hermitian transpose represents the adjoint operation holds only for the \(\Mprod\)-dot product. A $\Mprod$-linear operator is also linear, and as such has a adjoint with respect to the usual dot product. That adjoint is not necessarily represented by the $\Mprod$-Hermitian transpose. We will revisit this point later.

\section{Practical Tubal Algebra}
\label{sec:org5c7caf2}

We now introduce a tool that not only makes computations in tubal algebra easier, but also is a valuable tool for theoretical work. The idea is to introduce a \emph{transformed domain} of tensors, in which the \(\Mprod\)-product separates.

First, let's define the (frontal) facewise-product between two tensors. 
Given a tensor \(\tenA \in \R^{m \times p \times n}\) and a tensor \(\tenB \in \R^{p \times l \times n}\), the facewise-product of \(\tenA\) and \(\tenB\), denoted by \(\tenA \faceprod \tenB\), is the tensor \(\tenC \in \R^{m \times l \times n}\) whose \(j\)th frontal slices is equal to the product of \(j\)th frontal slice of \(\tenA\) and \(j\)th frontal slice of \(\tenB\). 
That is, \(\matC_j = \matA_j \matB_j\), where $\matA_j,\matB_j$ and $\matC_j$ denote the $j$th frontal slices in $\tenA,\tenB$ and $\tenC$ respectively. 
Next, given a tensor \(\tenA \in \R^{m \times p \times n}\), we define the transformed tensor \(\hat{\tenA}\in\C^{m\times p \times n}\) as \(\hat{\tenA} \coloneqq \tenA \nmodeprod{3} \matM\). 

Suppose that \(\tenA\) and \(\tenB\) are tubes, i.e., in \(\R^{1\times 1 \times n}\).
Both identified with a \(\KM\) element, and we denote by \(\va\in\R^n\) and \(\vb\in\R^n\) the corresponding vectors. Let \(\tenC = \tenA \Mprod \tenB\), which is also a tube \(\R^n \ni \vc = \va \Mprod \vb\). The transformed tensors \(\hat{\tenA}, \hat{\tenB}, \hat{\tenC}\) are also tubes and correspond to complex vectors \(\hat{\va}=\matM\va, \hat{\vb}=\matM\vb, \hat{\vc}=\matM\vc\). Since \(\vc = \va \Mprod \vb = \matM^{-1}(\matM \va \hadprod \matM \vb)\), we have \(\hat{\vc} = \hat{\va} \hadprod \hat{\vb}\). Twisting into tubes, the Hadamard product becomes facewise product, and we have \(\hat{\tenC} = \hat{\tenA} \faceprod \hat{\tenB}\). We have shown the last equality only for tensors that have \(1\times 1\) frontal slices, but due to the way the \(\Mprod\)-product is defined for general tensors, this result immediately extends to general tensors. 

\begin{fact}
For tensors \(\tenA, \tenB, \tenC\) of compatible sizes:
\begin{equation*}
\tenC = \tenA \Mprod \ten B \quad \iff \quad \hat{\tenC} = \hat{\tenA} \faceprod \hat{\tenB}
\end{equation*}
\end{fact}

This result immediately leads to a simple and efficient way to implement the \(\Mprod\)-product between two tensors: move to the transform domain, do facewise products, and then move back. Thus reducing the computation of the $\Mprod$-product to a mostly parallel set of matrix-matrix multiplications, which are extremely efficient.  But the utility of this result goes much further. It allows us to work separately on each frontal slices, after we converted everything to the transformed domain. Only at the end, we move back to the primal domain. Deriving the following two examples is much easier via this method. 

\begin{example}[Identity tensor]
The tensor
\begin{equation*}
\tenI_m = \left[\begin{array}{cccc}
		\unitM &        &        &        \\
		       & \unitM &        &        \\
		       &        & \ddots &        \\
		       &        &        & \unitM 
		\end{array}
	\right]
\end{equation*}
is the identity tensor of size \(m\). That is \(\tenI_m \Mprod \tenA = \tenA\), and \(\tenB \Mprod \tenI_m = \tenB\) for any tensors \(\tenA\) and \(\tenB\) of compatible sizes. The frontal slices of the transformed tensor \(\hat{\tenI}_m\) are the identity matrices of size \(m\). That is, \(\hat{\matI}_m = \matI_m\). Indeed, \(\hat{\ve}_{\matM} = \matM \ve_{\matM} = \matM \matM^{-1} \mathbf{1} = \mathbf{1}\), so each diagonal entry of \(\hat{\tenI}_m\) is \(1\).
\end{example}

\begin{example}[Hermitian transpose in the transform domain]
For a tensor \(\tenX\) denote by \(\tenX^{\ha\Delta}\) the tensor obtained by the Hermitian conjugate of the frontal slices of \(\tenX\). Then, \(\widehat{\tenX^\ha} = \hat{\tenX}^{\ha \Delta}\).
\end{example}
\section{Tubal SVD and Eckart-Young-like Low-rank Tensor Approximations}
\label{sec:org134fd1b}

One of the main advantages of the tubal algebra is that it is ``matrix mimetic''. That is, many of the results for matrices have a direct counterpart in the tubal algebra where you replace matrix multiplication with \(\Mprod\)-product (sometimes small additional modifications are needed). One of the most glaring example is the existence of \emph{tubal SVD} (t-SVD for short), and its use in low-rank tensor approximations. First, we need two definitions. 

\begin{definition}
A tensor \(\tenU \in \KM^{m\times m}\) is a \(\Mprod\)-unitary tensor if \(\tenU^\ha \Mprod \tenU = \tenU \Mprod \tenU^\ha = \tenI_m\). 
\end{definition}
Note that a tensor \(\tenU\) is \(\Mprod\)-unitary if and only if its transformed tensor \(\hat{\tenU}\) is facewise unitary (in the usual sense). That is, \(\hat{\tenU}^\ha \faceprod \hat{\tenU} = \hat{\tenU} \faceprod \hat{\tenU}^\ha = \hat{\tenI}_m\).

\begin{definition}
A tensor is \emph{f-diagonal} if all its frontal slices are diagonal.
\end{definition}

\begin{theorem}[Theorem 5.1 in~\cite{kernfeld-2015-tensor-tensor}]
Let \(\tenA \in \KM^{m\times p}\) be a tensor. Then there exist \(\Mprod\)-unitary tensors \(\tenU \in \KM^{m\times m}\) and \(\tenV \in \KM^{p\times p}\) and an f-diagonal tensor \(\tenS \in \KM^{m\times p}\) such that \(\tenA = \tenU \Mprod \tenS \Mprod \tenV^\ha\). Furthermore, if we denote by \(\vsigma_1, \dots, \vsigma_{\min(m,p)}\) the diagonal elements of \(\tenS\), then \(\vsigma_1 \Mgeq \vsigma_2 \Mgeq \cdots \Mgeq \vsigma_{\min(m,p)}\Mgeq 0\).
We refer to such a decomposition as the \emph{\(\Mprod\)-SVD} (called ``t-SVDM'' in \cite{kilmer-2021-tensor-tensor}) of \(\tenA\), and \(\vsigma_1, \dots, \vsigma_{\min(m,p)}\) as the \emph{\(\Mprod\)-singular tubes} of \(\tenA\).
\end{theorem}

\begin{proof}
Consider \(\hat{\tenA}\). Construct a SVD decomposition for each frontal slice on its own, collecting the left and right unitary matrices to form \(\hat{\tenU}\) and \(\hat{\tenV}\) respectively, and the singular values in \(\hat{\tenS}\). Note that each frontal slice \(\hat{\matU}_j\) and \(\hat{\matV}_j\) is computed independently, and they are not required to be the same unitary matrix across slices.
With this construction, we have \(\hat{\tenA} = \hat{\tenU} \faceprod \hat{\tenS} \faceprod \hat{\tenV}^{\ha \Delta}\). We can then move back to the primal domain to obtain \(\tenA = \tenU \Mprod \tenS \Mprod \tenV^\ha\). Since, by construction, each frontal slice of \(\hat{\tenU}\) and \(\hat{\tenV}\) is unitary, \(\tenU\) and \(\tenV\) are \(\Mprod\)-unitary. The f-diagonal property of \(\tenS\) follows from the f-diagonal property of \(\hat{\tenS}\).

As for the ordering of the singular values, first note that \(\hat{\tenS}\) is a real tensor (since it contains only singular values). Also, in each frontal slice, the diagonal elements are non-increasing. This immediately implies the ordering in the theorem statement. 
\end{proof}

A worked example illustrating the t-SVD computation on a small \(2\times 2\times 2\) tensor is given in Appendix~\ref{app:worked-examples} (Example~\ref{ex:tsvd-worked}).

\subsection{Low-rank Tensor Approximations via Truncated Tubal SVD}
\label{sec:org1e2e623}

One of the most powerful results on matrix SVD is the Eckart-Young theorem, which shows that we can build an optimal low-rank approximation of a matrix by truncating its SVD decomposition. A similar result holds for tubal SVD, if we restrict \(\matM\) to be unitary up to scaling. But first, we need an appropriate notion of rank. 

\begin{definition}
The \emph{\(\Mprod\)-rank} of a tensor \(\tenA\), denoted by \(\Mrank(\tenA)\), is the maximal \(r\) such that \(\vsigma_r\neq 0\). 
\end{definition}

\begin{lemma}
The \(\Mprod\)-rank of a tensor is the maximal rank of the frontal slices of its transformed tensor: \(\Mrank(\tenA) = \max_j \rank(\hat{\matA}_j)\).
\end{lemma}

\begin{proof}
The tube \(\vsigma_i\) is non-zero if and only if \(\hat{\vsigma}_i\) is non-zero. \(\hat{\vsigma}_i\) is non-zero if and only if there exists an index \(j\) in which it is non-zero, which means that the frontal slice \(\hat{\matA}_j\) has rank at least \(i\). The result follows. 
\end{proof}

We also need the following result.
\begin{lemma}[Theorem 3.1 in~\cite{kilmer-2021-tensor-tensor}]
Suppose \(\matM=c\matW\) for some unitary matrix \(\matW\) and non-zero scalar \(c\). Let \(\tenU\in\KM^{m\times m}\) be a \(\Mprod\)-unitary tensor. Then for any tensor \(\tenB\) and \(\tenC\) of compatible sizes, we have \(\Fnorm{\tenU \Mprod \tenB} = \Fnorm{\tenB}\) and \(\Fnorm{\tenC \Mprod \tenU} = \Fnorm{\tenC}\).
\end{lemma}

\begin{proof}
We have 
\begin{equation*}
\Fnorm{\hat{\tenB}} = \Fnorm{\tenB \nmodeprod{3} \matM} = \Fnorm{c\matW \matB_{(3)}} = |c|\cdot \Fnorm{\matB_{(3)}} = |c|\cdot \Fnorm{\tenB}
\end{equation*}
Let \(\tenD = \tenU \Mprod \tenB\). We have,
\begin{align*}
\FnormS{\tenB} &= \frac{1}{|c|^2}\FnormS{\hat{\tenB}} = \frac{1}{|c|^2}\sum_{i=1}^n \FnormS{\hat{\matB}_i} = \frac{1}{|c|^2}\sum_{i=1}^n \FnormS{\hat{\matU}_i \hat{\matB}_i} \\
&= \frac{1}{|c|^2}\sum_{i=1}^n \FnormS{\hat{\matD}_i} = \frac{1}{|c|^2}\FnormS{\hat{\tenD}} = \FnormS{\tenD}
\end{align*}
Proof for the second part is similar.
\end{proof}

\begin{corollary}[Corollary 3.3 in~\cite{kilmer-2021-tensor-tensor}]
Consider the \(\Mprod\)-SVD of a tensor \(\tenA\in\KM^{m\times p}\) given by \(\tenA = \tenU \Mprod \tenS \Mprod \tenV^\ha\). If \(\matM=c\matW\) for some unitary matrix \(\matW\) and non-zero scalar \(c\) then \(\FnormS{\tenA} = \FnormS{\tenS}=\sum_{i=1}^{\Mrank(\tenA)} \FnormS{\vsigma_i}\). Moreover, \(\Fnorm{\vsigma_1} \geq \Fnorm{\vsigma_2} \geq \cdots\).
\end{corollary}

\begin{proof}
\(\FnormS{\tenA} = \FnormS{\tenS}\) follows from the previous lemma, while
\[
\FnormS{\tenS}=\sum_{i=1}^{\Mrank(\tenA)} \FnormS{\vsigma_i}
\]
follows from the definition of the Frobenius norm since \(\tenS\) is diagonal with the singular tubes on the diagonal.
As for the second part of the corollary, since \(\matM = c\matW\), we have \(\Fnorm{\vsigma_j} = \frac{1}{|c|}\Fnorm{\hat{\vsigma}_j}\) for \(j=1,\dots,\min(m,p)\). Clearly, for any two tubes \(\va\) and \(\vb\) of the same size \(\va \Mgeq \vb\) implies \(\Fnorm{\hat{\va}} \geq \Fnorm{\hat{\vb}}\), and combined with the equality \(\Fnorm{\vsigma_j} = \frac{1}{|c|}\Fnorm{\hat{\vsigma}_j}\) we get the desired result.
\end{proof}

Now, we state a first Eckart-Young-like theorem for the tubal tensor algebra.
\begin{theorem}[Theorem 3.7 in \cite{kilmer-2021-tensor-tensor}]
\label{thm:first-eckart-young}
Assume that \(\matM=c\matW\) for some unitary matrix \(\matW\) and non-zero scalar \(c\). Let \(\tenA\in\KM^{m\times p}\) be a tensor and let \(\tenA = \tenU \Mprod \tenS \Mprod \tenV^\ha\) be its \(\Mprod\)-SVD. For any \(k\leq \Mrank(\tenA)\), the tensor \(\lrapprox{\tenA}{k} \coloneqq \tenU_{:,1:r,:} \Mprod \tenS_{1:r,1:r,:} \Mprod \tenV_{:,1:r,:}^\ha\) is the optimal \(\Mprod\)-rank-\(k\) approximation of \(\tenA\) in the Frobenius norm. That is, for any tensor \(\tenB\) of the same size as \(\tenA\) and \(\Mrank(\tenB)\leq k\), we have \(\FnormS{\tenA - \lrapprox{\tenA}{k}} \leq \FnormS{\tenA - \tenB}\). Moreover, \(\FnormS{\tenA - \lrapprox{\tenA}{k}} = \sum_{i=k+1}^{\Mrank(\tenA)} \FnormS{\vsigma_i}\).
\end{theorem}

\begin{proof}
The square error result follows easily from the previous corollary. 

Let \(\tenB\) be a tensor of the same size as \(\tenA\) and \(\Mrank(\tenB)\leq k\). To show \(\FnormS{\tenA - \lrapprox{\tenA}{k}} \leq \FnormS{\tenA - \tenB}\) it is enough to show that this holds for each frontal slice on its own. However, since \(\matM=c\matW\) for some unitary matrix \(\matW\) and non-zero scalar \(c\), this is equivalent to showing it holds for frontal slices of the transformed tensors. Since the transformation is linear, we are left with showing that \(\FnormS{\hat{\matA}_j - (\widehat{\lrapprox{\tenA}{k}})_{::j}} \leq \FnormS{\hat{\matA}_j - \hat{\matB}_j}\) for \(j=1,\dots,n\).  Due to the linearity of the transformation, we have \(\hat{\lrapprox{\tenA}{k}} = \hat{\tenU}_{:,1:k,:} \faceprod \hat{\tenS}_{1:k,1:k,:} \faceprod \hat{V}^{\ha\Delta}_{:,1:k,:}\). For an index \(j\), we have \((\widehat{\lrapprox{\tenA}{k}})_{::j} = \hat{\tenU}_{:,1:k,j} \hat{\tenS}_{1:k,1:k,j}  \hat{V}^{\ha}_{:,1:k,j}\) where we are now using regular matrix multiplication (and implicitly assuming the singleton dimension is squeezed out). Due to the way \(\Mprod\)-SVD is built, the right hand side is a truncated SVD of \(\hat{\matA}_j\). This shows that \((\widehat{\lrapprox{\tenA}{k}})_j = \lrapprox{\hat{\matA}_j}{k}\), where here the square brackets denote matrix rank truncations. Since \(\Mrank(\tenB) \leq k\) we have \(\rank(\hat{\matB}_j)\leq k\). The Eckart-Young theorem now implies that
\begin{equation*}
\FnormS{\hat{\matA}_j - (\widehat{\lrapprox{\tenA}{k}})_j} = \FnormS{\hat{\matA}_j - \lrapprox{\hat{\matA}_j}{k}} \leq \FnormS{\hat{\matA}_j - \hat{\matB}_j}
\end{equation*}
\end{proof}

In the first Eckart-Young theorem we allowed only the same truncation level over all different frontal slices. We can allow more refined truncation. The corresponding Eckart-Young result is less elegant, but more useful. 
\begin{definition}
The \emph{\(\Mprod\)-multirank} of a tensor \(\tenA\),\linebreak denoted by \(\Mmultirank(\tenA)\), is \(n\)-sized tuple of non-negative integers in which the \(j\)th entry is equal to the rank \(\hat{\matA}_j\). Given a \(n\)-sized tuple of non-negative integers \(\vr\), we say that a tensor \(\tenX\) has \(\Mprod\)-multirank at most \(\vr\) if \(\Mmultirank(\tenA) \leq \vr\) entrywise. We say that a tuple \(\vr\) \emph{respects the conjugation pattern of \(\matM\)} if \(r_i = r_j\) for every conjugate pair of rows \(i,j\) in \(\matM\). Note that \(\Mmultirank(\tenA)\) respects the conjugation pattern of \(\matM\).
\end{definition}

\begin{theorem}[Theorem 3.8 in~\cite{kilmer-2021-tensor-tensor}]
\label{thm:second-eckart-young}
Assume that \(\matM=c\matW\) for some unitary matrix \(\matW\) and non-zero scalar \(c\). Let \(\tenA\in\KM^{m\times p}\) be a tensor and let \(\tenA = \tenU \Mprod \tenS \Mprod \tenV^\ha\) be its \(\Mprod\)-SVD. For any \(n\)-size tuple of non-negative integers \(\vr \leq \Mmultirank(\tenA)\) which respects the conjugation pattern of \(\matM\), we define the tensor \(\lrapprox{\tenA}{\vr}\) as the tensor in which the frontal domain \(j\) in the transformed space is the \(r_j\) truncated SVD of the \(j\)th frontal face of the transformed \(\tenA\). That is,
\begin{equation*}
(\widehat{\lrapprox{\tenA}{\vr}})_{::j} \coloneqq \lrapprox{\hat{\matA}_j}{r_j} = \hat{\matU}_{:,1:r_j,j} \hat{\matS}_{1:r_j,1:r_j,j} \hat{\matV}^\ha_{:,1:r_j,j}
\end{equation*}
where in the above we use regular matrix multiplication (and implicitly assume the singleton dimension is squeezed out). \(\lrapprox{\tenA}{\vr}\) is the the optimal \(\Mprod\)-multirank-\(k\) approximation of \(\tenA\) in the Frobenius norm. That is, for any tensor \(\tenB\) of the same size as \(\tenA\) and \(\Mprod\)-multirank at most \(\vr\), we have \(\FnormS{\tenA - \lrapprox{\tenA}{\vr}} \leq \FnormS{\tenA - \tenB}\). Moreover, \(\FnormS{\tenA - \lrapprox{\tenA}{\vr}} = |c|^{-1} \cdot \sum_{j=1}^n\sum_{i=r_j+1}^{\rank(\hat{\matA}_j)} \hat{\vsigma}^2_{ij}\).
\end{theorem}

\begin{proof}
The proof is essentially the same as the proof of the previous theorem, so we omit it.
\end{proof}

A worked example illustrating multirank truncation, building on the t-SVD computation referenced above, is given in Appendix~\ref{app:worked-examples} (Example~\ref{ex:truncation-worked}).

\subsection{Hilbert Algebra Structure}

Even though the t-SVD exists for {\em every} tubal ring, Eckart-Young-like optimality was proved only for $\matM$ being proportional to a unitary matrix. In this subsection we seek to understand what distinguishes tubal rings with scaled unitary $\matM$ from ones that do not. In Appendix~\ref{app:hilbert-algebra} we show that for scaled unitary $\matM$ (and a slightly larger class of matrices), the tubal ring carries a so-called ``Hilbert algebra'' structure~\cite{Takesaki2003}: the adjoint, with respect to the Frobenius inner product, of the multiplication operator $T_\va(\vx)\coloneqq\va\Mprod\vx$ is given by multiplication by the conjugate $\va^\sconj$. Though we have not formally linked this Hilbert algebra structure to the Eckart-Young-like results (Theorems~\ref{thm:first-eckart-young} and~\ref{thm:second-eckart-young}), we conjecture that the Hilbert structure is what enables them. 

\section{Conclusions}

The goal of this paper was to show that tubal tensor algebra could be derived from first principles given the tubal precept: view a third-order tensor as a matrix of vectors (called tubes). We showed how to construct  $\Mprod$ using a series of logical steps, and that the resulting construction was unavoidable. The resulting product is the only way to define a ring of tubes with all the desired properties (commutative, unital and von Neumann regular). In doing so, we aim to partially demystify tubal tensor algebra: tubal tensor algebra is the only construction that translates the tubal precept to a well-founded mathematical theory with desired properties.

At the same time, our paper does not answer a yet unresolved key  question regrading $\Mprod$: which $\matM$ should be used, and why some choices of $\matM$ work better than others (in applications)? We do know, empirically, that  FFT-like transforms and wavelet transforms work better than random orthogonal ones~\cite{kilmer-2021-tensor-tensor}, but there is no theoretical explanation for this. Connected is recent literature on learning the $\matM$ in $\Mprod$ from data~\cite{elizabeth2024a}.

Another interesting topic to consider is the utility of relaxing some of the requirements of the tubal product. Keegan and Newman recently introduced a matrix-mimetic tensor algebra with Eckart-Young-like optimality that relaxes the invertibility requirement on $\matM$ with the requirement for a matrix $\matQ$ with orthonormal columns~\cite{keegan2024a}. In followup work, the second author established sufficient and necessary conditions Eckart-Young-like optimality result of tubal tensors~\cite{mor-2025-sufficientnecessary}. In another work, Eckart-Young-like optimality was shown for ring groups even with non-abelian groups, which are  non-commutative rings.  We leave complete characterization of Eckart-Young optimality of non-commutative rings to future research.

\section{Proof of Theorem~\ref{thm:main}}
\label{sec:proof}

Throughout this section, we assume that \((\R^n, +, \cdot)\) is a tubal ring. Since tubal rings are von Neumann regular, we denote for every \(\va\in\R^n\) by \(\va^\winv\) the unique element for which \(\va = \va\cdot\va^{\winv}\cdot\va\) and \(\va^\winv=\va^{\winv}\cdot\va\cdot\va^{\winv}\), eschewing the previous definition in Eq.~\eqref{eq:def-winv} which requires the existence of \(\matM\) a-priori. Once we prove Theorem~\ref{thm:main} the two definitions coincide.

We also make use of the following result:
\begin{lemma}[Fact 5.17.8 in \cite{bernstein-2011-matrix-mathem}]
\label{lem:joint-diag-commutative}
    Let $S \subseteq \C^{n\times n}$, and assume every matrix $\matA \in S$ is diagonalizable over $\C$. Then, $\matA \matB = \matB \matA$ for all $\matA,\matB\in S$ if and only if there exists a nonsingular $\matS\in \C^{n\times n}$ such that, for all $\matA\in\C^{n\times n}$, $\matS\matA\matS^{-1}$ is diagonal.
\end{lemma}

\subsection{Sub-algebra of Matrices}

We begin showing that there is a natural way to encode {\em any} tubal ring as a sub-algebra of \(n\times n\) matrices. 

Consider a fixed \(\va \in \R^n\). The map \(\vx \mapsto \va \cdot \vx\) is linear (since \(\cdot\) is bilinear). Thus, there is a unique matrix \(\matR_\va \in \R^{n\times n}\) for which \(\matR_\va \vx = \va \cdot \vx\) for all \(\vx\).
\begin{definition}
    \label{def:rep-mat}
    Given \(\va \in \R^n\), the unique matrix \(\matR_\va \in \R^{n\times n}\) for which \(\matR_\va \vx = \va \cdot \vx\) for all \(\vx\) is called the {\em representation matrix of $\va$}. For a set $S\subseteq \R^n$, we use \(\Rep(S)\) to denote the set of all representative matrices for elements of $S$, i.e.,
    \[
    \Rep(S) \coloneqq \left\{\matR_\va\,:\,\va\in S \right\}
    \]
\end{definition}
\begin{example}[Representation matrix for \(\Mprod\)]\label{example:vetmat.identification.mprod}
    Consider the \(\cdot\) defined by Eq.~\eqref{eq:Mprod} for some \(\matM\in\C^{n\times n}\). Then,
    \begin{equation*}
        \matR_\va = \matM^{-1} \diag(\matM \va) \matM
    \end{equation*}
\end{example}

The following sequence of lemmas establish that the map \(\va \mapsto \matR_\va\) completely defines a linear-subspace and sub-algebra of \(\R^{n \times n}\) matrices.

\begin{lemma}
    \(\matR_{\vzero_{n}} = \matZero_{n\times n}\) where \(\vzero_n\) is the $n$ sized vector of zeros, and \(\matZero_{n \times n}\) is the \(n\times n\) matrix of zeros. There is no non-zero $\vx$ for which $\matR_\vx = \matZero_{n\times n}$.
\end{lemma}
\begin{proof}
   The element  $\vzero_{n \times 1}$ is the additive identity of the ring, making it\linebreak \textit{the} (unique) ring's zero element.
    As such, $\vzero_n$ is an absorbing element, i.e., $\vzero_n \cdot \vb = \vzero_n$,  for any $\vb \in \R^n$. Thus, we have $\matR_{\vzero_n} \vb = \vzero_n$ for all $\vb$, so $\matR_{\vzero_n}$ must be the zero matrix. Furthermore, if $\vx\neq\vzero_n$ then there must be a some $\vy$ such that $\vx\cdot\vy\neq\vzero_n$, so $\matR_{\vx}\vy\neq\vzero_n$ so we must have $\matR_{\vx}\neq\matZero_{n\times n}$.
\end{proof}

\begin{lemma}
    The map \(\va \mapsto \matR_\va\) is injective. 
\end{lemma}
\begin{proof}
    Suppose \(\matR_\va = \matR_\vb\). Due to associativity, for all \(\vx\)
    \begin{equation*}
    \matR_{\va - \vb} \vx = (\va - \vb)\cdot \vx = \va \cdot \vx - \vb \cdot \vx = \matR_\va \vx - \matR_\vb \vx = (\matR_\va - \matR_\vb)\vx = \vzero_n
    \end{equation*}
    Since this holds for all \(\vx\) we conclude that \(\matR_{\va - \vb} = \matZero_{n\times n}\). Thus, \(\va - \vb = \vzero_n\) which implies that \(\va = \vb\).
\end{proof}

\begin{lemma}
\label{lem:rep-linear}
    For any \(\alpha \in \R\) and \(\vx,\vy \in \R^n\):
    \begin{enumerate}
        \item \(\matR_{-\vx} = -\matR_{\vx}\)
        \item \(\matR_{\vx + \vy} = \matR_\vx + \matR_\vy\)
        \item \(\matR_{\alpha \vx} = \alpha \matR_\vx\)
        \item If \(\ve\) denotes the multiplicative identity element in the ring, \(\matR_\ve = \matI_n\).
        \item \(\matR_{\vx \cdot \vy} = \matR_\vx \matR_\vy\)
    \end{enumerate}
\end{lemma}
\begin{proof}

    \begin{enumerate}
        \item Since a tubal ring is an algebra, we have for every $\va$, $\matR_{-\vx}\va=(-\vx)\cdot\va=-(\vx\cdot\va) = -\matR_{\vx}\va$. This holds for all $\va$, so $\matR_{-\vx}=-\matR_{\vx}$.
        \item For every $\va$, $\matR_{\vx+\vy}=(\vx + \vy)\cdot\va = \vx \cdot \va + \vy \cdot \va = \matR_{\vx}\va + \matR_{\vy}\va=(\matR_{\vx} + \matR_{\vy})\va$. So, \(\matR_{\vx + \vy} = \matR_\vx + \matR_\vy\).
        \item Since a tubal ring is an algebra over $\R$, we have for every $\va$, $\matR_{\alpha\vx}\va=(\alpha\vx)\cdot\va=\alpha(\vx\cdot\va) = \alpha\matR_{\vx}\va$. So, \(\matR_{\alpha \vx} = \alpha \matR_\vx\).
        \item For every $\va$, $\matR_{\ve}\va = \ve \cdot \va = \va$, so $\matR_{\ve}$ is the identity matrix.
        \item For every $\va$, $\matR_{\vx\cdot\vy} = (\vx \cdot \vy)\cdot \va = \vx \cdot (\vy \cdot \va) = \matR_{\vx} \matR_\vy \va$, so \(\matR_{\vx \cdot \vy} = \matR_\vx \matR_\vy\). 
    \end{enumerate}
\end{proof}
\begin{corollary}
    \(\Rep(\R^n)\) is linear subspace of \(\R^{n\times n}\), and also an algebra (so is a sub-algebra of $\R^{n\times n}$). 
\end{corollary}

The followings play an important role in our proof.
\begin{definition}
    An element $\va\in\R^n$ is called {\em idempotent} if $\va \cdot \va = \va$.  
\end{definition}

\begin{lemma}
    $\va\in\R^n$ is idempotent if and only if the matrix $\matR_\va$ is idempotent.
\end{lemma}
\begin{proof}
    Obvious from Lemma~\ref{lem:rep-linear}.
\end{proof}

\subsection{Diagonalizability of Tubal Rings}

The multiplication \(\cdot\) is defined between two real vectors. However, it will be useful to work also with complex operands. Thus, for \(\va \in \R^n\) and \(\vb\in\C^n\) we define:
\[
    \va \cdot \vb \coloneqq \va \cdot \rpart\vb + i (\va \cdot \ipart\vb)
\]
We still have \(\va \cdot \vb = \matR_\va \vb\).

\begin{definition}
    We say that a non-zero \(\vx\in\C^n\) is an {\em eigenvector of \(\va\)} if there exists corresponding {\em eigenvalue} \(\lambda\in\C\) such that 
    \[
        \va \cdot \vx  = \lambda \vx
    \]
    Equivalently, \((\lambda, \vx)\) is an eigenpair of \(\matR_\va\).
\end{definition}

\begin{definition}
    We say that \(\va\) is {\em diagonalizable} if there exists set of \(n\) linearly independent eigenvectors of \(\va\).
\end{definition}
\begin{lemma}
    \(\va\) is diagonalizable if and only if \(\matR_\va\) is diagonalizable over \(\C\). The set of eigenvectors of \(\va\) is exactly the set of eigenvectors of \(\matR_\va\).
\end{lemma}
\begin{proof}
    Immediate from the definition of \(\matR_\va\).
\end{proof}

We are really interested in cases where all elements are diagonalizable, and that there is a set of \(n\) linearly independent vectors which are eigenvectors for all vectors in \(\R^n\).
\begin{definition}
    We say that $S\subseteq \R^n$ is {\em diagonalizable} if every \(\va\in S\) is diagonalizable. We say that the tubal ring is diagonalizable if $\R^n$ is diagonalizable.
    
    We say that the \(S \subseteq \R^n\) is {\em jointly diagonalizable} if there exists set of \(n\) linearly independent vectors which are eigenvectors for {\em all}  \(\va\in S\). We say that the tubal ring is jointly diagonalizable if $\R^n$ is jointly diagonalizable.
\end{definition}

\begin{lemma}
    $S \subseteq \R^n$ is jointly diagonalizable if and only if the set of matrices $\Rep(S)$ is jointly diagonalizable. 
\end{lemma}
\begin{proof}
    Again, immediate from the definition of representation matrices (Definition~\ref{def:rep-mat}).
\end{proof}
\begin{lemma}
\label{lem:diag-to-joint-diag}
    In a tubal ring, if a set of elements $S$ is diagonalizable then it is jointly diagonalizable.  
\end{lemma}
\begin{proof}
    To show that $S$ is jointly diagonalizable we need to show that $\Rep(S)$ is jointly diagonalizable. Since $S$ is diagonalizable, every matrix in $\Rep(S)$ is diagonalizable. Since every matrix in $\Rep(S)$ is a representative matrix, and the ring is commutative, then for $\matA,\matB\in\Rep(S)$ we have $\matA\matB=\matB\matA$. Thus Lemma~\ref{lem:joint-diag-commutative} implies that there is a non-singular $\matS$ that diagonalizes all matrices in $\Rep(S)$, i.e., $\Rep(S)$ is jointly diagonalizable.
\end{proof}
\begin{lemma}
    \label{lem:span-diag-is-diag}
    If $S\subseteq \R^n$ is diagonalizable, then $\spn(S)$ is diagonalizable.
\end{lemma}
\begin{proof}
    $S$ is diagonalizable, so $\Rep(S)$ is jointly diagonalizable. Thus, there exists a $\matS$ such that for every $\va\in S$ we have \(
    \matS \matR_\va \matS^{-1} = \matLambda_\va
    \) for some diagonal matrix \(\matLambda_\va\). Consider a finite linear combination of elements in $S$: $\vb = \sum^N_{j=1}\alpha_j \vb_j\in\spn(S)$. Then, $\matR_\vb = \sum^N_{j=1}\alpha_j \matR_{\vb_j}$. Now, 
    \[
    \matS \matR_\vb \matS^{-1}= \sum^N_{j=1}\alpha_j \matS \matR_{\vb_j}\matS^{-1}= \sum^N_{j=1}\alpha_j \matLambda_{\vb_j} 
    \]
    is diagonal, so $\vb$ is diagonalizable. We have shown that any finite linear combination is diagonalizable, so $\spn(S)$ is diagonalizable. 
\end{proof}

We now come to the main result of this subsection. It shows that we only need to show that the tubal ring is diagonalizable.
\begin{proposition}
    \label{prop:diag-is-Mprod}
    A tubal ring is diagonalizable if and only if there exists an invertible \(\matM \in \C^{\times n}\) such that for all \(\va,\vb\in\R^n\),
    \[
        \va \cdot \vb = \matM^{-1}(\matM \va \hadprod \matM \vb)
    \]
\end{proposition}
\begin{proof}
    Due Lemma~\ref{lem:diag-to-joint-diag} there exists a matrix \(\matS\in\C^{n\times n}\) that diagonalizes \(\matR_\va\) for all \(\va\in\R^n\). Thus, for every \(\va\) there exists a diagonal matrix \(\matLambda_\va\) such that
    \[
    \matS \matR_\va \matS^{-1} = \matLambda_\va
    \]
    Consider the map \(\va \mapsto \matS \matR_\va \matS^{-1} 1_{n\times 1}\). It is easy to see that this map is linear (this is due to the fact that the map \(\va \mapsto \matR_\va\) is linear), so there exist a matrix \(\matM\in\C^{n\times n}\) that implements it. That is \(\matM \va = \matS \matR_\va \matS^{-1} 1_{n\times 1}\) for all \(\va\). This, in turn, implies that \(\matLambda_\va = \diag(\matM \va)\) and \(\matR_\va = \matS^{-1} \diag(\matM \va) \matS\) for all \(\va\). Now, for all \(\va,\vb\): 
    \[
        \va \cdot \vb = \matR_\va \vb = \matS^{-1} \diag(\matM \va) \matS \vb =\matS^{-1}(\matM \va \hadprod \matS\vb)
    \]

    Let \(\ve\) denote the unit element in the ring, and set \(\vbeta = \matS \ve\). For every \(\vx\):
    \begin{align*}
        \vx &= \vx \cdot \ve \\
            &= \matS^{-1}(\matM \vx \hadprod \matS \ve) \\
            &= \matS^{-1}(\matS \ve \hadprod \matM \vx) \\
            &= \matS^{-1} \diag(\vbeta) \matM\vx 
    \end{align*}
    Since this holds for all \(\vx\) we conclude that \(\matS^{-1} \diag(\vbeta) \matM = \matI_n\). Thus, all elements in \(\vbeta\) must be non-zero (for otherwise the left hand side is rank deficient), and \(\matS = \diag(\vbeta) \matM\). That is, \(\matS\) is a row-rescaling of \(\matM\). However, close inspection of the formula \(\va \cdot \vb =\matS^{-1}(\matM \va \hadprod \matS\vb)\) reveals that it still holds if we rescale the rows of \(\matS\) with non-zero factors. Thus, we conclude that \(\va \cdot \vb = \matM^{-1}(\matM \va \hadprod \matM \vb)\) for all \(\va,\vb\).

    The other direction is immediate.
\end{proof}

\subsection{Sub-rings, Ideals and Idempotent Elements}

Our technique for proving diagonalizability is to decompose the ring into a direct sum of sub-rings, and then show each one of them is diagonalizable. Then, the full tubal ring is diagonalizable since it taking the union of the sub-rings spans all of $\R^n$. The decomposition is based on John von Neumann's theory of regular rings~\cite{JvN36}. Some of the results are special cases of results therein, but we include a proof to make the text self-contained as possible.

\begin{definition}
    Let $V\subseteq \R^n$ be a linear subspace. We say that \((V, +, \cdot)\) is a {\em tubal sub-ring} if it is commutative, unital, von Neumann regular ring, which is also an associative algebra over \(\R\), and $V$ is an ideal in \((\R^n,+,\cdot)\)\footnote{This means that if $\vx\in V$ and $\va \in \R^n$ then \(\va\cdot\vx\in V\).}. The dimension of the sub-ring is $\dim(V)$.
\end{definition}

One way to construct a tubal sub-ring is via the principal ideal associated with (idempotent) elements. It is, in fact, the only way.

\begin{definition}
    For $\va\in\R^n$, the {\em principal ideal associated with $\va$} is:
    \[
    (\va) \coloneqq \left\{\vx \cdot \va\,:\,\vx\in \R^n \right\}
    \]
\end{definition}
\begin{lemma}
    \(((\va), +, \cdot)\) is a tubal sub-ring.
\end{lemma}
\begin{proof}
    We start with linearity. Suppose $\vx,\vy\in (\va)$ and $\alpha,\beta\in\R$. We can write $\vx = \vx' \cdot \va$ and $\vy = \vy' \cdot \va$ for some $\vx',\vy'\in\R^n$. Then, $\alpha\vx + \beta\vy = (\alpha\vx' + \beta\vy')\cdot\va$ so in $(\va)$, where we used both distributivity of $\cdot$ with respect to $+$, and the fact that $(\R^n,+,\cdot)$ is an associative algebra. So $(\va)$ is a linear subspace. It remains to show that $(\va)$ is closed under $\cdot$ to establish it is a ring. This follows immediately from $\cdot$ being commutative. 

    We show that it is unital by identifying a unit element. Let $\ve_\va \coloneqq \va \cdot \va^\winv$. Notice that $\ve_\va \cdot \va = \va \cdot \va^\winv \cdot \va = \va$. We now show that for every $\vx \in (\va)$ we have $\ve_\va \cdot \vx = \vx$. Indeed, if $\vx \in (\va)$ we can write $\vx = \vx' \cdot \va$ for some $\vx'$. Now, $\ve_\va \cdot \vx = \ve_\va \cdot \vx' \cdot \va = \ve_\va \cdot \va \cdot \vx' = \va \cdot \vx' = \vx$. 

    Finally, we show that or every $\vx \in (\va)$ we also have $\vx^\winv \in (\va)$. Recall that $\vx^\winv = \vx^\winv \cdot \vx \cdot \vx^\winv$. Since $(\va)$ is an ideal, products on the left and right, even with elements not necessarily from $(\va)$, must be in $(\va)$, so $\vx^\winv\in (\va)$.
\end{proof}

\begin{lemma}
    If \((V, +, \cdot)\) is a tubal sub-ring, then there is a unique idempotent \(\ve_V \in \R^n\) such that $V=(\ve_V)$. 
\end{lemma}
\begin{proof}
    Since the sub-ring is unital, it has a unique unit element $\ve_V$. That element must be idempotent, since it being the unit we must have $\ve_V \cdot \ve_V = \ve_V$. Now, for every $\va\in V$ we have \(\va \cdot \ve_V = \va\) so $V \subseteq (\ve_V)$. For the other direction, if $\va \in (\ve_V)$, we have that $\va = \vx \cdot \ve_V$ for some $\vx$, but by assumption $V$ is an ideal, so $\va\in V$ since $\ve_V\in V$. 
\end{proof}

Given a tubal sub-ring $V$, and an idempotent element $\vp \in V$, we can decompose $V$ into a direct sum of two principal ideals. If $\vp$ is non-trivial, each one of these will have a dimensional smaller than $V$'s.
\begin{lemma}
\label{lem:decompose-from-idem}
    Suppose $(V,+,\cdot)$ is a tubal sub-ring, and denote by $\ve_V$ its unique unit element. Then for every idempotent $\vp \in V$ we have 
    \[
    V = (\vp)\oplus (\ve_V - \vp)
    \]
\end{lemma}
\begin{proof}
    We first show that $(\vp) \cap (\ve_V - \vp) = \{0_{n}\}$. Suppose that $\vx \in (\vp) \cap (\ve_V - \vp)$. Since $\vx \in (\vp)$ and $\vp$ is idempotent, we have $\vp \cdot \vx = \vx$. The is because we can write $\vx = \vp \cdot \vy$ for some $\vy$, and then $\vp \cdot \vx = \vp \cdot \vp \cdot \vy = \vp \cdot \vy = \vx$. Likewise, since $\ve_V - \vp$ is idempotent as well, $(\ve_V - \vp)\cdot\vx = \vx$. However, $(\ve_V - \vp)\cdot\vx=(\ve_V - \vp)\cdot\vp\cdot\vx=(\ve_V \cdot \vp -\vp \cdot \vp)\cdot\vx = \vzero_n\cdot\vx = \vzero_n$. So $\vx=\vzero_n$.
    
    Suppose $\vx \in V$. Then $\vx = \vp \cdot \vx + (\ve_V - \vp)\cdot\vx$, so $\vx \in (\vp)\oplus (\ve_V - \vp)$. This shows that $V \subseteq (\vp)\oplus (\ve_V - \vp)$. On the other-hand, if $\vx \in (\vp)\oplus (\ve_V - \vp)$ then there exist \(\vy,\vz\in\R^n\) such that $\vx=\vp\cdot\vy + (\ve_V - \vp)\cdot\vz$. Now, since $V$ is an ideal in $(\R^n,+,\cdot)$, and $\vp,\ve_V-\vp\in V$, then both summands are in $V$. Since $V$ is a linear subspace, this implies that $\vx\in V$. So, $(\vp)\oplus (\ve_V - \vp)\subseteq V$. Together, $V=(\vp)\oplus (\ve_V - \vp)$.
\end{proof}

\subsection{Decomposing to a Direct Sum Tubal Sub-Fields}

Let \(V\) be a tubal sub-ring (possibly \(\R^n\)). It has two trivial idempotent elements: \(\vzero_n\) and \(\ve_V\). However, for both these trivial idempotent elements the decomposition given by Lemma~\ref{lem:decompose-from-idem} is useless, since one of the principal ideals is $V$ while the other is the trivial subspace $\{\vzero_n\}$. However, if we can find a non-trivial idempotent $\vp\in V$, then both principal ideals $(\vp)$ and $(\ve_V - \vp)$ are proper subsets of $V$, and so have lower dimension.

Recall that a ring is called a \emph{division ring} if every non-zero element has a multiplicative inverse. We now show that if a tubal ring $V$ is {\em not} a division ring, then it has at least one non-trivial idempotent element. 
\begin{lemma}
    \label{lem:not-field-nontrivial-idem}
    Suppose a tubal sub-ring $V$ is not a division ring. Then, there exist a non-trivial idempotent element in $V$.
\end{lemma}
\begin{proof}
    Since $V$ is not a division ring, it has a non-zero element $\va\in V$ which does not have a multiplicative inverse in the ring. Consider $\vp = \va \cdot \va^{\winv}$. As shown earlier, it is idempotent. However, since $\va$ does not have a multiplicative inverse in the ring, then $\va \cdot \va^{\winv} \neq \ve_V$. The identity $\va = \va \cdot \va^{\winv} \cdot \va = \vp \cdot \va$ ensures that $\vp$ is non-zero, so it is a non-trivial idempotent element in $V$. 
\end{proof}

A unital, commutative, division ring is a field, so we call a tubal sub-ring that is also a division ring a {\em tubal sub-field}. An immediate corollary of the previous lemma is that a tubal sub-ring can be decomposed as a sum of tubal sub-fields.
\begin{corollary}
    \label{cor:full-decompose-to-ideals}
    Let $V$ be a tubal sub-ring. There exists idempotent elements $\vp_1,\dots,\vp_N \in \R^n$, for a finite $N$, such that 
    \[
    V = (\vp_1)\oplus (\vp_2) \oplus \cdots \oplus (\vp_N)
    \]
    and \((\vp_j)\) are tubal sub-fields for all $j=1,\dots,N$. 
\end{corollary}
\begin{proof}
    By induction on the dimension of $V$. If $V$ has dimension 1, it is not only the principal ideal of $\ve_V$, but also spanned by $\ve_V$, i.e., every element is of the form $\alpha \ve_V$ for scalar $\alpha\in\R$. The element is $\vzero_n$ if $\alpha=0$, in which case it doesn't need to have a multiplicative inverse for $V=(\ve_V)$ to be a tubal sub-field. If $\alpha\neq 0$, then the multiplicative inverse of $\alpha\ve_V$ inside $V$ is $\alpha^{-1} \ve_V$. We see that every element has a multiplicative inverse, so $V=(\ve_V)$ is a tubal sub-field. 

    For a general $n=\dim(V)$, if $V=(\ve_V)$ is a tubal sub-field we are done. Otherwise, by Lemma~\ref{lem:not-field-nontrivial-idem} it has a non-trivial idempotent $(\vp)$, and by Lemma~\ref{lem:decompose-from-idem} we can decompose $V=(\vp)\oplus(\ve_V - \vp)$. Each of the two principal ideals in the direct sum is a non-trivial subspace of $V$, so both have dimension smaller than $n$, and by induction they have a decomposition.
\end{proof}

\subsection{Diagonalizability of Tubal Sub-Fields}
We now recall the Frobenius theorem on possible dimensions of associative division algebras (Theorem~\ref{thm:fro}). Note that a tubal sub-ring is an associative algebra, so if it is also a division ring, it is an associative division algebra. 
Thus, if $V$ is a tubal sub-field it must be isomorphic to either $\R$ or $\C$, which are easily diagonalizable:
\begin{proposition}
    \label{prof:sub-fields-diag}
    A tubal sub-field is diagonalizable.
\end{proposition}
\begin{proof}
    A tubal sub-field is a field, so it must be isomorphic to $\R$ or $\C$ (it cannot be isomorphic to $\mathbb{H}$ since $\mathbb{H}$ is not commutative). So, the dimension of $V$ must be 1 or 2.

    If the dimension is 1, then we must have $V=\{\alpha \ve_V\,|\,\alpha\in\R\}$. So, the representative matrices must be of the form $\alpha\matR_{\ve_V}$, so we only need to show that $\matR_{\ve_V}$ is diagonalizable. Since $\ve_V$ is idempotent, $\matR_{\ve_V}$ is idempotent as well, so it must be diagonalizable. 

    If the dimension is 2, then the field must be isomorphic to $\C$. In the isomorphism, $1$ is identified with $\ve_V$, while the imaginary unit $\si$ is identified with some element $\vi_V \in V$ for which $\vi_V \cdot \vi_V = -\ve_V$. Every element of $V$ is a linear combination of $\ve_V$ and $\vi_V$, so it is enough to show that both are diagonalizable. The sub-field unit $\ve_V$ is idempotent, so its representative matrix is idempotent, so diagonalizable. 

    We are left with showing that $\vi_V$, i.e., $\matR_{\vi_V}$ is diagonalizable. Apart from $\matR_{\ve_V}$ being idempotent, we also have $\matR^2_{\vi_V}=-\matR_{\ve_V}$ and $\matR_{\vi_V}\matR_{\ve_V}=\matR_{\ve_V}\matR_{\vi_V}=\matR_{\vi_V}$ where the last equality follows from the fact that $\vi_V\in V$ and $\ve_V$ is an identity with respect to $V$. Since $\matR_{\ve_V}$ is idempotent, we can decompose $\R^n = \range(\matR_{\ve_V})\oplus\nullsp(\matR_{\ve_V})$. Since $\matR_{\vi_V} \matR_{\ve_V}=\matR_{\vi_V}$ we have for every $\vx\in\nullsp(\matR_{\ve_V})$, $\matR_{\vi_V}\vx = 0$, so $\matR_{\vi_V}$ behaves like the zero matrix on that subspace, and the associated minimal polynomial is $x$. Now, on $\range(\matR_{\ve_V})$ we have that $\matR_{\vi_V}^2$ behaves like minus the identity (since $\matR_{\ve_V}$ is the identity on that subspace). So the minimal polynomial of $\matR_{\vi_V}$ divides $x^2+1$ on that subspace. Overall, we obtain that the minimal polynomial of $\matR_{\vi_V}$ divides $x(x^2+1)$ which factorizes $x(x+\si)(x-\si)$. All factors are distinct linear factors, so $\matR_{\vi_V}$ is diagonalizable.
\end{proof}

\subsection{Bringing It All Together}

Corollary~\ref{cor:full-decompose-to-ideals} ensures we can decompose $\R^n$ to a direct sum of tubal sub-fields. Proposition~\ref{prof:sub-fields-diag} ensures that each are diagonalizable. Lemma~\ref{lem:span-diag-is-diag} show that the span of the union of all those ideals is diagonalizable, but that union is clearly $\R^n$. Finally, Proposition~\ref{prop:diag-is-Mprod} concludes the proof.

\subsection{Additional Discussion}

Additional discussion of Theorem~\ref{thm:main} can be found in Appendix~\ref{app:additional-discussion}, including a sketch of a shorter (but less elementary) proof, a remark on the role of weak inverses, and an algorithm for recovering $\matM$ from a tubal product.

\section{Additional Proofs}
\label{sec:proof-winv-realness}

\begin{lemma}
\label{lem:struct-Minv}
    Suppose \(\matM\in\C^{n\times n}\) is an invertible matrix for which every row is either real, or is conjugate to exactly one other row. Consider $\matM^{-1}$. If row $j$ in $\matM$ is real, then column $j$ of $\matM^{-1}$ is real as well. On the other hand, if row $j$ is complex and conjugate to row $k$, then columns $j$ and $k$ in $\matM^{-1}$ are complex and conjugate to each other.
\end{lemma}
\begin{proof}
    Without loss of generality we can assume that the first $p$ rows of $\matM$ are real, and that the conjugate pairs are $(p+1,p+2),(p+3,p+4),\dots,(n-1,n)$. If that was not the case, we could just permute the rows, and apply the same permutation to $\matM^{-1}$.

    Consider the matrix
    \begin{equation*}
        \matZ = \begin{bmatrix}
            1  &   &        &   &     &      &        &     &   \\
               & 1 &        &   &     &      &        &     &   \\
               &   & \ddots &   &     &      &        &     &   \\ 
               &   &        & 1 &     &      &        &     &   \\
               &   &        &   & 1   & 1    &        &     &   \\
               &   &        &   & \si & -\si &        &     &   \\
               &   &        &   &     &      & \ddots &     &   \\
               &   &        &   &     &      &        & 1   & 1 \\
               &   &        &   &     &      &        & \si & -\si               
        \end{bmatrix}
    \end{equation*}
    where the leading identity is of size $p$, and the block $\begin{bmatrix}
        1 & 1 \\ \si & -\si \end{bmatrix}$ is repeated $(n-p)/2$ times. The matrix $\matB = \matZ\matM$ is real, and as  the product of two invertible matrices is invertible as well. The inverse of a real matrix, if it exists, is real, so $\matB^{-1} = \matM^{-1}\matZ^{-1}$ is real. From $\matM^{-1}= \matB^{-1}\matZ$ and realness of $\matB^{-1}$ the claim follows.
\end{proof}

\begin{lemma}
    Suppose \(\matM\in\C^{n\times n}\) is an invertible matrix for which every row is either real, or is conjugate to exactly one other row. Then $\matM^{-1}\overline{\matM}$ is real. 
\end{lemma}
\begin{proof}
    Again, without loss of generality we may assume the real rows of $\matM$ are the first rows, while all paired complex rows appear consecutively afterwards. Writing $\matB = \matZ\matM$ like in the previous proof, we have $\matM^{-1}\overline{\matM}=\matB^{-1}\matZ\overline{\matZ}^{-1} \matB$ where we used the fact that $\matB$ is real. Since, 
    \begin{equation*}
        \begin{bmatrix}
            1 & 1 \\
            \si & -\si
        \end{bmatrix}\overline{
        \begin{bmatrix}
            1 & 1 \\
            \si & -\si
        \end{bmatrix}}^{-1} =\
        \begin{bmatrix}
            1 & 0 \\ 0 & -1
        \end{bmatrix}
    \end{equation*}
    we see that $\matZ\overline{\matZ}^{-1}$ is a real matrix. So, $\matM^{-1}\overline{\matM}$ is the product of three real matrices, so must be real. 
\end{proof}

\begin{lemma}
    Let  \(\matM\in\C^{n\times n}\) be an invertible matrix. The result of Eq. \eqref{eq:Mprod} is real for every $\va,\vb\in \R^n$, if and only if every row of $\matM$ is either real, or is conjugate to exactly one other row of $\matM$.
\end{lemma}

\begin{proof}
    Suppose every row of $\matM$ is either real, or is conjugate to exactly one other row of $\matM$. Let $\va,\vb\in \R^n$. Consider $\matM\va$. Every element in $\matM\va$ is either guaranteed to be real (if the corresponding row in $\matM$ is real), or is either guaranteed to have a conjugate element in another index (the index of the row that is conjugate to the row corresponding to the element). Furthermore, exactly the same structure is present in $\matM\vb$, so this structure carries to $\matM\va\hadprod\matM\vb$. When multiplying on the left by $\matM^{-1}$, due to Lemma~\ref{lem:struct-Minv}, each entry guaranteed to be real meets a real column, while guaranteed conjugate pairs meet conjugate columns, so we get a contribution that is the sum of conjugate vectors so is real. Thus, the result is the sum of real vectors, so it must be real.
    
    We now show the other direction. Assume that $\matM$ is such that the result of Eq. \eqref{eq:Mprod} is real for every $\va,\vb\in \R^n$. For $\vx \in \C^n$, denote  $\matR_{\vx}\coloneqq \matM^{-1} \diag(\matM \vx) \matM$. For all $\vx,\vy\in\C^n$ we have $\vx \Mprod \vy = \matR_\vx \vy$. For $\va\in\R^n$ we have $\matR_\va \vy = \va\Mprod\vy\in\R^n$ for all $\vy\in\R^n$, so $\matR_\vx$ must be real, and this holds for all $\vx\in\R^n$.
    So, $\matR_\vx=\overline{\matR_\vx}$, i.e.,
    \begin{equation*}
         \matM^{-1} \diag(\matM \vx) \matM =  \overline{\matM}^{-1} \diag(\overline{\matM} \vx) \overline{\matM}
    \end{equation*}
    Let $\matN \coloneqq \overline{\matM}\matM^{-1}$. Multiply on the left by $\overline{\matM}$ and on the right by $\matM^{-1}$ to get 
    \begin{equation*}
        \matN \diag(\matM\vx) = \diag(\overline{\matM}\vx)\matN
    \end{equation*}
    For an index $k$, let $\ve_k$ denote the $k$th standard basis vector. Noticing that $\matM\ve_k = \matM_{:k}$ we see that 
    \begin{equation*}
        \label{eq:MN}
        \matN \diag(\matM_{:k}) = \diag(\overline{\matM}_{:k})\matN
    \end{equation*}
    for $k=1,\dots,n$. 
    
    Suppose $i,j$ are indices for which $\matN_{ij}\neq 0$. The $(i,j)$ entry on the left of Eq.~\eqref{eq:MN} is equal to $\matN_{ij}\matM_{jk}$, while the $(i,j)$ entry on the right is equal to $\overline{\matM}_{ik}\matN_{ij}$, both of which are equal. We see that $\matM_{jk} = \overline{\matM}_{ik}$ and this holds for all for all $k$, so the entries $i$th and $j$th rows of $\matM$ are conjugate of each other. 

    Consider a row index $i$. $\matN$ is the product of two invertible matrices, so it is invertible. This means that row $i$ is not zero. Furthermore, it has a single non-zero entry. To see this, suppose row $i$ has two non-zero entries: $\matN_{ij}\neq 0$ and $\matN_{il}\neq 0$. Row $j$ of $\matN$ is conjugate of row $i$, and likewise row $l$  of $\matN$ is conjugate of row $i$. We conclude that rows $j$ and $l$ are both conjugate of row $i$, so they must be equal, which contradicts the assumption that $\matM$ is invertible. So, row $i$ of $\matM$ has a single non-zero entry $\matN_{ij}$ for some $j$, and row $i$ and $j$ of $\matM$ are conjugate of each other. If $j=i$, then row $i$ is conjugate of itself, so it must be real. If $j\neq i$, then row $i$ is complex and has a conjugate pair in $\matM$. Since each row is non-zero in some entry, we have shown the claim.
\end{proof}

\begin{proof}[Proof of Lemma~\ref{lem:winv-realness}]
Consider index $j$. If row $j$ of $\matM$ is real, then define $\vp_j \coloneqq \matM^{-1}\ve_j$ (i.e., column $j$ of $\matM^{-1}$). This vector is real. On the other hand, if it is complex, and paired with row $k>j$ (if $k<j$, the simple exchange them), define $\vp_j \coloneqq \matM^{-1}(\ve_j + \ve_k)$ and  $\vp_k \coloneqq \matM^{-1}(\si\ve_j - \si\ve_k)$. Both of these are real vectors, due to Lemma~\ref{lem:struct-Minv} .

Overall, we have defined $n$ vectors $\vp_1, \dots, \vp_n$. They are linearly independent because the columns or $\matM^{-1}$ are linearly independent. Consider two indexes $j$ and $k$. We have the following cases regarding $\vp_j \Mprod \vp_k$:
\begin{itemize}
    \item If $j=k$ and row $j$ in $\matM$ is real. In this case, 
    \begin{equation*}
        \vp_j \Mprod \vp_j = \matM^{-1}(\matM\vp_j \hadprod \matM\vp_j) = \matM^{-1}(\ve_j \hadprod \ve_j) = \matM^{-1}\ve_j = \vp_j 
    \end{equation*}
    \item If $j=k$ and row $j$ in $\matM$ is complex. In this case, there is a conjugate pair $l$. If $j < l$:
    \begin{equation*}
         \vp_j \Mprod \vp_j = \matM^{-1}(\matM\vp_j \hadprod \matM\vp_j)= \matM^{-1}((\ve_j + \ve_l) \hadprod (\ve_j+\ve_l)) =  \matM^{-1}(\ve_j + \ve_l) = \vp_j
    \end{equation*}
    If $j > l$:
      \begin{align*}
         \vp_j \Mprod \vp_j &= \matM^{-1}(\matM\vp_j \hadprod \matM\vp_j) \\
         &= \matM^{-1}((\si\ve_l - \si\ve_j) \hadprod (\si\ve_l - \si\ve_j)) =  -\matM^{-1}(\ve_l + \ve_j) = -\vp_l
       \end{align*}
    \item If $j\neq k$, and both row $j$ and $k$ in $\matM$ are real. In this case,
    \begin{equation*}
        \vp_j \Mprod \vp_k = \matM^{-1}(\matM\vp_j \hadprod \matM\vp_k) = \matM^{-1}(\ve_j \hadprod \ve_k) = \vzero_n 
    \end{equation*} 
    \item If $j\neq k$ and row $j$ is real, while row $k$ complex, then $\matM\vp_j = \ve_j$ is non-zero only in index $j$, while $\matM\vp_k$ is non-zero only in indices $k$ and $l$. Regardless, we see that $\matM\vp_j\hadprod\matM\vp_k=\vzero_n$, so $$\vp_j \Mprod \vp_k = \vzero_n$$
    \item If $j\neq k$, and both row $j$ and $k$ in $\matM$ are complex. If they are not conjugate pairs, then $\matM\vp_j$ and $\matM\vp_k$ are non-zero in different indices, and so 
    $$\vp_j \Mprod \vp_k = \vzero_n$$
    If they are conjugate pairs, assume without loss of generality that $j < k$, and
        \begin{align*}
         \vp_j \Mprod \vp_k &= \matM^{-1}(\matM\vp_j \hadprod \matM\vp_k) \\
         &= \matM^{-1}((\ve_j + \ve_k) \hadprod (\si\ve_j-\si\ve_k)) =  \matM^{-1}(\si\ve_j-\si\ve_k) = \vp_k
    \end{align*}
\end{itemize}
We also have the following cases regarding the $\KM$-conjugate of $\vp_j$ for $j=1,\dots,n$:
\begin{itemize}
    \item If row $j$ is real in $\matM$, then $\matM\vp_j$ is real as well, and so $\vp^\sconj_j = \vp_j$.
    \item If row $j$ is complex in $\matM$, and its conjugate pair $k$ has a higher index ($k>j$), then also $\matM\vp_j$ is real as well, and so $\vp^\sconj_j = \vp_j$.
    \item If row $j$ is complex in $\matM$, and its conjugate pair $k$ has a lower index ($k<j$) then
    \begin{equation*}
        \vp^\sconj_j=\matM^{-1}\overline{\matM \vp_j} = \matM^{-1}\overline{(\si\ve_k - \si\ve_j)} = 
        \matM^{-1}(\si\ve_j - \si\ve_k) = -\vp_j
    \end{equation*}
\end{itemize}

Given a vector $\va$, write it $\va = \sum^n_{j=1} \gamma_j \vp_j$ where $\gamma_1,\dots,\gamma_n\in\R$. We now define another vector $\vy=\sum_{i=1}\delta_j \vp_j\in\R^n$ via the following rules:
\begin{itemize}
    \item If row $j$ is real in $\matM$, then $\delta_j=\gamma^\pinv_j$.
    \item If rows $j$ and $k$ are conjugate pair of complex rows in $\matM$ with $j<k$, and $\gamma_k\gamma_l\neq0$ then 
    \begin{align*}
        \delta_j &= \frac{\gamma_j}{\gamma^2_j +\gamma^2_k} \\
        \delta_k &= \frac{-\gamma_k}{\gamma^2_j +\gamma^2_k}
    \end{align*}
\end{itemize}
It is not hard, though a bit tedious, to verify that the following hold:
\begin{align*}
    \va \Mprod \vy \Mprod \va  &= \va \\
    \vy \Mprod \va \Mprod \vy  &= \vy \\
    (\va \Mprod \vy)^\sconj &= \va \Mprod \vy \\
    (\vy \Mprod \va)^\sconj &= \vy \Mprod \va
\end{align*}
The same four equations hold if we replace $\va$ with $\matR_\va$, $\vy$ with $\matR_\vy$ and $\Mprod$ with matrix product. Thus, $\matR_\vy$ is the Moore-Penrose pseudoinverse of $\matR_\va$. Now, notice that even if $\va^\winv$ is complex, the same four equations hold with $\va^
\winv$ replacing $\vy$, and similary $\matR_{\va^\winv}$ is the Moore-Penrose pseudoinverse of $\matR_\va$. However, the Moore-Penrose pseudoinverse is unique, so $\matR_{\va^\winv}= \matR_\vy$. Via the definitions of $\matR_\vy$ and $\matR_{\va^\winv}$, it is immediate that this implies that $\va^\winv=\vy$. Since $\vy$ is real, so is $\va^\winv$.

\end{proof}

\subsection{Canonical Tubal Rings and the Proof of Proposition \ref{prop:realness-iso}}

Given $n$ and $m\leq n$ of same parity, define
    \begin{equation*}
        {\renewcommand{\arraystretch}{0.85}\setlength{\arraycolsep}{2pt}
        \matM_{n,m} \coloneqq \begin{bmatrix}
            1  &   &        &   &     &      &        &     &   \\
               & 1 &        &   &     &      &        &     &   \\
               &   & \ddots &   &     &      &        &     &   \\
               &   &        & 1 &     &      &        &     &   \\
               &   &        &   & \nicefrac{1}{\sqrt{2}}   & \nicefrac{-\si}{\sqrt{2}}  &        &     &   \\
               &   &        &   & \nicefrac{1}{\sqrt{2}} & \nicefrac{\si}{\sqrt{2}} &        &     &   \\
               &   &        &   &     &      & \ddots &     &   \\
               &   &        &   &     &      &        & \nicefrac{1}{\sqrt{2}}   & \nicefrac{-\si}{\sqrt{2}} \\
               &   &        &   &     &      &        & \nicefrac{1}{\sqrt{2}} & \nicefrac{\si}{\sqrt{2}}
        \end{bmatrix}}
    \end{equation*}
where the number of leading 1's is $m$. We call $\mathbb{K}_{n,m}\coloneqq \mathbb{K}_{\matM_{n,m}}$ for various $n,m$ the {\em canonical tubal rings}. We show that each tubal ring is isomorphic to a canonical tubal ring. Proposition \ref{prop:realness-iso} is an immediate corollary.

\begin{proposition}
    Consider the tubal ring defined by $\matM\in\C^{n\times n}$. Let $m$ be the realness of $\matM$. Then $\KM$ is isomorphic as a ring to $\mathbb{K}_{n,m}$.
\end{proposition}

\begin{proof}
    We can assume with out loss of generality that the real rows of $\matM$ are the first $m$ rows, and that the conjugate rows appear in adjunct pairs afterwards. Otherwise, we can simply permute the entries for this to hold, and this is clearly an isomorphic operation.

    Note that
    \begin{equation*}
        {\renewcommand{\arraystretch}{0.85}\setlength{\arraycolsep}{2pt}
        \matM^{-1}_{n,m} \coloneqq \begin{bmatrix}
            1  &   &        &   &     &      &        &     &   \\
               & 1 &        &   &     &      &        &     &   \\
               &   & \ddots &   &     &      &        &     &   \\
               &   &        & 1 &     &      &        &     &   \\
               &   &        &   & \nicefrac{1}{\sqrt{2}}   & \nicefrac{1}{\sqrt{2}}  &        &     &   \\
               &   &        &   & \nicefrac{\si}{\sqrt{2}} & \nicefrac{-\si}{\sqrt{2}} &        &     &   \\
               &   &        &   &     &      & \ddots &     &   \\
               &   &        &   &     &      &        & \nicefrac{1}{\sqrt{2}}   & \nicefrac{1}{\sqrt{2}} \\
               &   &        &   &     &      &        & \nicefrac{\si}{\sqrt{2}} & \nicefrac{-\si}{\sqrt{2}}
        \end{bmatrix}}
    \end{equation*}
    so $\matM' \coloneqq \matM^{-1}_{n,m} \matM$ is an invertible real matrix. Let $T:\R^n\to\R^n$ be defined by $T(\vx)=\matM' \vx$. We claim that $T$ is the isomorphism between $\KM$ and $\mathbb{K}_{n,m}$ are isomorphic. Since $T$ is a linear operation, we need only to verify it respects that product operation and the conjugation operation. That is, for every $\vx,\vy$ we have $T(\vx \Mprod \vy) = T(\vx) \star_{\matM_{n,m}} T(\vy)$ and $T(\vx^\sconj) = T(\vx)^\sconj$, where in the second equality the left conjugation is with respect to $\matM$ and the right is with respect to $\matM_{n,m}$.
    For the first equation,
    {\setlength{\jot}{1pt}
    \begin{align*}
        T(\vx \Mprod \vy) &= T(\matM^{-1}(\matM\vx \hadprod \matM\vy))
        = \matM^{-1}_{n,m}(\matM\vx \hadprod \matM\vy)) \\
        &= \matM^{-1}_{n,m}( \matM_{n,m}\matM'\vx \hadprod \matM_{n,m}\matM'\vy))
        = \matM^{-1}_{n,m}( \matM_{n,m}T(\vx) \hadprod \matM_{n,m}T(\vy))) \\
        &= T(\vx) \star_{\matM_{n,m}} T(\vy).
    \end{align*}}
    For the second equation,
    {\setlength{\jot}{1pt}
    \begin{align*}
        T(\vx^\sconj) &= T(\matM^{-1}\overline{\matM \vx})
        = \matM^{-1}_{n,m}\overline{\matM \vx}
        = \matM^{-1}_{n,m}\overline{\matM_{n,m}\matM' \vx}
        = T(\vx)^\sconj.
    \end{align*}}
\end{proof}
\section*{Acknowledgments}

This research was funded by the Israel Science Foundation (grant 1524/23). AI was used in preparing this paper, predominantly for copy editing and formal verification of proofs in Lean (the proof themselves were written by the authors). The authors assume responsibility for all content. 

\label{bibliography link}
\bibliography{tensors}

\clearpage

\appendix

\begin{center}
{\Large\bfseries Appendix: Demystifying Tubal Tensor Algebra}
\end{center}
\bigskip

\section{Worked Examples}
\label{app:worked-examples}

In this appendix we present detailed worked examples that illustrate the t-SVD construction and its truncation, complementing the theoretical development in Section~\ref{sec:org134fd1b}.

\begin{example}[Worked example: t-SVD of a $2\times 2\times 2$ tensor]
\label{ex:tsvd-worked}
We illustrate the t-SVD computation for a small tensor \(\tenA \in \R^{2\times 2 \times 2}\) with \(\matM = \matF_2 = \left[\begin{smallmatrix} 1 & 1 \\ 1 & -1 \end{smallmatrix}\right]\), the DFT matrix for \(n=2\). Here \(\matF_2^{-1} = \tfrac{1}{2}\matF_2\), so the tubal product is the t-product \(\tprod\).

\medskip
\noindent\textbf{Step 1: The input tensor.}
Define \(\tenA\) by its two frontal slices:
\[
\matA_1 = \begin{bmatrix} 3 & 1 \\ 1 & 3 \end{bmatrix}, \qquad
\matA_2 = \begin{bmatrix} 1 & 1 \\ 1 & 1 \end{bmatrix}.
\]

\noindent\textbf{Step 2: Move to the transform domain.}
Compute \(\hat{\tenA} = \tenA \nmodeprod{3} \matF_2\). Since \(\matF_2 = \left[\begin{smallmatrix} 1 & 1 \\ 1 & -1 \end{smallmatrix}\right]\), the transformed frontal slices are:
\[
\hat{\matA}_1 = \matA_1 + \matA_2 = \begin{bmatrix} 4 & 2 \\ 2 & 4 \end{bmatrix}, \qquad
\hat{\matA}_2 = \matA_1 - \matA_2 = \begin{bmatrix} 2 & 0 \\ 0 & 2 \end{bmatrix} = 2\matI.
\]

\noindent\textbf{Step 3: Compute SVD of each transformed frontal slice.}
The matrix \(\hat{\matA}_1\) is symmetric with eigenvalues \(6\) and \(2\) (eigenvectors \(\tfrac{1}{\sqrt{2}}\left[\begin{smallmatrix} 1 \\ 1\end{smallmatrix}\right]\) and \(\tfrac{1}{\sqrt{2}}\left[\begin{smallmatrix} 1 \\ -1\end{smallmatrix}\right]\)):
\[
\hat{\matU}_1 = \hat{\matV}_1 = \matQ \coloneqq \frac{1}{\sqrt{2}}\begin{bmatrix} 1 & 1 \\ 1 & -1\end{bmatrix},\quad
\hat{\matS}_1 = \begin{bmatrix} 6 & 0 \\ 0 & 2 \end{bmatrix}.
\]
The matrix \(\hat{\matA}_2 = 2\matI\) has both singular values equal to \(2\), so its SVD is not unique. We are free to choose any orthogonal factors; for consistency we set \(\hat{\matU}_2 = \hat{\matV}_2 = \matQ\), \(\hat{\matS}_2 = 2\matI\).

\medskip
\noindent\textbf{Step 4: Assemble the transform-domain factors.}
The singular tubes are
\(\hat{\vsigma}_1 = \left[\begin{smallmatrix} 6 \\ 2 \end{smallmatrix}\right]\) and
\(\hat{\vsigma}_2 = \left[\begin{smallmatrix} 2 \\ 2 \end{smallmatrix}\right]\).
Both are non-zero, and \(\hat{\vsigma}_1 \geq \hat{\vsigma}_2\) entry-wise, confirming the ordering \(\vsigma_1 \Mgeq \vsigma_2\).

\medskip
\noindent\textbf{Step 5: Move back to the primal domain.}
Since \(\hat{\matU}_1 = \hat{\matU}_2 = \matQ\), we get \(\matU_1 = \matQ\), \(\matU_2 = \matZero\) (and similarly \(\matV_1 = \matQ\), \(\matV_2 = \matZero\)).
For the singular factor:
\[
\tenS:\quad \matS_1 = \tfrac{1}{2}(\hat{\matS}_1 + \hat{\matS}_2) = \begin{bmatrix} 4 & 0 \\ 0 & 2 \end{bmatrix},\quad
\matS_2 = \tfrac{1}{2}(\hat{\matS}_1 - \hat{\matS}_2) = \begin{bmatrix} 2 & 0 \\ 0 & 0 \end{bmatrix}.
\]
The primal singular tubes are \(\vsigma_1 = \left[\begin{smallmatrix} 4 \\ 2 \end{smallmatrix}\right]\) and \(\vsigma_2 = \left[\begin{smallmatrix} 2 \\ 0 \end{smallmatrix}\right]\). Since both are non-zero, \(\tenA\) has \(\tprod\)-rank \(2\).

\medskip
\noindent\textbf{Step 6: Verification.}
In the transform domain, \(\hat{\matU}_j\hat{\matS}_j\hat{\matV}_j^\T = \hat{\matA}_j\) for \(j=1,2\) by construction, confirming \(\tenA = \tenU \tprod \tenS \tprod \tenV^\T\). Note that while the frontal slices of \(\tenU\) are not individually unitary (\(\matU_2 = \matZero\)), the tensor \(\tenU\) is \(\tprod\)-unitary since each \(\hat{\matU}_j\) is unitary.
\end{example}

\begin{example}[Worked example: multirank truncation]
\label{ex:truncation-worked}
Continuing Example \ref{ex:tsvd-worked}, we truncate \(\tenA\) to \(\tprod\)-multirank \(\vr = (1,1)\). In the transform domain, we keep only the leading singular value from each slice:
\[
(\widehat{\lrapprox{\tenA}{\vr}})_1 = 6 \cdot \tfrac{1}{\sqrt{2}}\begin{bmatrix}1\\1\end{bmatrix} \cdot \tfrac{1}{\sqrt{2}}\begin{bmatrix}1\\1\end{bmatrix}^\T = \begin{bmatrix}3&3\\3&3\end{bmatrix}, \qquad
(\widehat{\lrapprox{\tenA}{\vr}})_2 = 2 \cdot \tfrac{1}{\sqrt{2}}\begin{bmatrix}1\\1\end{bmatrix} \cdot \tfrac{1}{\sqrt{2}}\begin{bmatrix}1\\1\end{bmatrix}^\T = \begin{bmatrix}1&1\\1&1\end{bmatrix}.
\]
Moving back to the primal domain, denoting the approximation by \(\tenB = \lrapprox{\tenA}{\vr}\):
\[
\matB_1 = \begin{bmatrix}2&2\\2&2\end{bmatrix}, \qquad
\matB_2 = \begin{bmatrix}1&1\\1&1\end{bmatrix}.
\]
The approximation error is \(\FnormS{\tenA - \tenB} = \FnormS{\left[\begin{smallmatrix} 1&-1\\-1&1\end{smallmatrix}\right]} + \FnormS{\left[\begin{smallmatrix} 0&0\\0&0\end{smallmatrix}\right]} = 4\), which equals the sum of the squared dropped singular values scaled by \(|c|^{-2}\): \((2^2 + 2^2)/2 = 4\).
\end{example}

\section{Additional Discussion of Theorem~\ref{thm:main}}
\label{app:additional-discussion}

In this appendix we collect supplementary observations about Theorem~\ref{thm:main} and the constructions in its proof: a sketch of a shorter (but less elementary) alternative proof, a remark on the role played by weak inverses, and an explicit algorithm extracted from the proof for recovering $\matM$ from a tubal product.

\paragraph{Shorter (But Less Elementary) Proof}
Our proof is based on showing that a tubal ring can be decomposed to a direct sum of fields. We made sure to use reasonably elementary concepts and results, instead of appealing to stronger ones from abstract algebra. A shorter argument is as follows. A tubal ring, being a ring over a finite dimensional vector space, is Artinian. Being von Neumann implies that it is also reduced (has no non-zero elements with square zero). A commutative Artinian ring that is reduced is semisimple, equivalently a finite product of fields. From here, writing the ring as a direct sum of fields is straightforward.

\paragraph{Importance of Existence of Weak Inverses}
A commutative ring is von Neumann regular if and only if reduced and has Krull dimension 0~\citep[Proposition 4.41]{BurklundEtAl22}. An Artinian ring always has Krull dimension 0 \citep[Theorem 8.5]{AtiyahMacdonald1969}. Thus, a commutative ring over $\R^n$ which is not tubal will be reduced, i.e., has an element $\vn$ such that $\vn^2=0$. Obviously, this is a somewhat pathological situation for elements that we want to serve as scalars. We conjecture that the existence of such elements precludes a meaningful definition for unitary tensors, with such definition essential for defining a tensor SVD. We leave this for future research.

\paragraph{Finding $\matM$}
Even though Theorem~\ref{thm:main} is stated as an existential result, an algorithm for finding $\matM$ given access only to the tubal product can be extracted from the proof. The procedure is given as Algorithm~\ref{alg:find-M} below.

If we are dealing with a tubal ring, Algorithm~\ref{alg:find-M} will find $\matM$. As for failing when dealing with a non-tubal ring, notice that the only step that might ``fail'' is the diagonalization of $\matR_\vx$ (line 7), and there is no guarantee that it indeed fails almost surely (thus certifying we are dealing with a non-tubal ring).

\begin{example}[Dual numbers are not a tubal ring - revisited]
    Going back to the example of dual numbers, notice that the representative matrix associated with $\vx=[x_1,x_2]^\T$ is
    \begin{equation*}
        \matR_\vx = \begin{bmatrix}
            x_1 & 0 \\
            x_2 & x_1
        \end{bmatrix}
    \end{equation*}
    which is not diagonalizable for $x_2$. So Algorithm~\ref{alg:find-M} will fail almost surely. This also proves that the dual numbers are not a tubal ring, as asserted in Subsection~\ref{subsec:examples}.
\end{example}

\begin{algorithm}[H]
\caption{\label{alg:find-M} Find Matrix $\matM$ for Tubal Ring Product}
\begin{algorithmic}[1]
\Require Binary operator $\mathsf{op}: \R^n \times \R^n \to \R^n$, integer $n$
\Ensure Matrix $\matM \in \R^{n \times n}$ such that $\mathsf{op}(\vx, \vy) = \matM^{-1} \left( (\matM \vx) \hadprod (\matM \vy) \right)$

\State $\vx \gets$ random vector in $\R^n$
\State Initialize $\matR_\vx \in \R^{n \times n}$ as zero matrix
\For{$i = 1$ to $n$}
    \State $\ve_i \gets$ $i$-th standard basis vector in $\R^n$
    \State $\matR_\vx[:, i] \gets \mathsf{op}(\vx, \ve_i)$
\EndFor
\State \textbf{Decompose:} $\matR_\vx = \matS \matLambda_\vx \matS^{-1}$ \Comment{eigendecomposition}
\State $\vy \gets \matS \ve$ \Comment{$\ve$ is the vector of ones}
\State Initialize $\matM \in \R^{n \times n}$ as zero matrix
\For{$i = 1$ to $n$}
    \State $\ve_i \gets$ $i$-th standard basis vector in $\R^n$
    \State Initialize $\matR_{\ve_i} \in \R^{n \times n}$ as zero matrix
    \For{$j = 1$ to $n$}
        \State $\ve_j \gets$ $j$-th standard basis vector in $\R^n$
        \State $\matR_{\ve_i}[:, j] \gets \mathsf{op}(\ve_i, \ve_j)$
    \EndFor
    \State $\matM[:, i] \gets \matS^{-1} \cdot \matR_{\ve_i} \cdot \vy$
\EndFor
\State \Return $\matM$
\end{algorithmic}
\end{algorithm}

\section{Hilbert Algebra Structure}
\label{app:hilbert-algebra}

In this appendix we establish the Hilbert algebra structure referenced in Section~\ref{sec:org134fd1b}.

Consider an element $\va \in \KM$, and define $T_\va:\KM \to \KM$ via $T_\va(\vx)=\va\Mprod \vx$. This is a linear operator, so it has an adjoint $T^\adj$ with respect to the inner product defined by the dot product, aka Frobenius inner product,  $\dotprod{\va}{\vb}\coloneqq \va^\T \vb$. We would expect that this adjoint will be given by the conjugate to $\va$, i.e., $T^\adj_\va(\vy)=T_{\va^\sconj}(\vy)=\va^\sconj \Mprod \vy$, however it is easy to build counter examples for which this does not hold. However, if $\matM$ is a scaled unitary matrix then it does hold, as the following proposition shows (it actually holds for a slightly larger class of matrices). We remark that the equality it proves is a key requirement of so-called ``Hilbert algebra'' structures~\cite{Takesaki2003}.

\begin{proposition}
    \label{prop:hilbert-algebra}
    Suppose that $\matM=\matD\matW$ for some invertible real diagonal matrix and unitary $\matW$. Then for every $\va,\vx,\vy$ we have
    \begin{equation}
    \label{eq:hilbert-algebra}
        \dotprod{\va\Mprod\vx}{\vy}=\dotprod{\vx}{\va^\sconj\Mprod \vy}.
    \end{equation}
\end{proposition}
\begin{proof}

    Note that for every $\vx$ we have $\va \Mprod \vx = \matM^{-1}\diag(\matM\va)\matM\vx$. Since $\matM=\matD\matW$ where $\matW$ is unitary and $\matD$ is diagonal, $\matM^{-1}=\matW^\ha \matD^{-1}$ and $\diag(\matM\va)=\matD\diag(\matW\va)$. So,
    \begin{equation}
    \label{eq:star-M-factored}
    \va \Mprod \vx = \matW^\ha\diag(\matW\va)\matD\matW\vx.
    \end{equation}
    Since $\va^\sconj=\matM^{-1}\overline{\matM}\va=\matW^\ha\overline{\matW}\va$, using same identities we have
    \begin{equation*}
    \va^\sconj \Mprod \vx = \matW^\ha\diag(\matW\va^\sconj)\matD\matW\vx =  \matW^\ha\diag(\overline{\matW}\va)\matD\matW\vx.
    \end{equation*}
    Thus,
    \begin{align*}
       \dotprod{\va\Mprod\vx}{\vy} &=  \dotprod{\matM^{-1}\diag(\matM\va)\matM\vx}{\vy} \\
        &= \left(\matM^{-1}\diag(\matM\va)\matM\vx\right)^\ha\vy \\
        &= \left(\matW^\ha\diag(\matW\va)\matD\matW\vx\right)^\ha\vy \\
        &= \vx^\ha \matW^\ha\matD\diag(\matW\va)^\ha\matW\vy \\
        &= \vx^\ha \matW^\ha\overline{\diag(\matW\va)}\matD\matW\vy \\
        &= \vx^\ha \matW^\ha\diag(\overline{\matW}\va)\matD\matW\vy \\
        &= \dotprod{\vx}{\va^\sconj\Mprod \vy}
    \end{align*}
\end{proof}

The Frobenius inner product can be extended to matrices and tensors. For $\tenA,\tenB\in\R^{I_1\times \dots \times I_N}$ define
\[
\dotprod{\tenA}{\tenB} \coloneqq \sum^{I_1}_{i=1}\cdots\sum^{I_N}_{i_N=1}\tenA_{i_1i_2\dots i_N} \tenB_{i_1i_2\dots i_N}
\]
Eq.~\eqref{eq:hilbert-algebra} can be extended tubal tensors:
\begin{proposition}
    \label{prop:hilbert-algebra-tensors}
    Suppose that $\tenA\in\KM^{m\times p}$. For every $\matX \in \KM^p$ and $\matY\in\KM^m$ we have
    \[
    \dotprod{\tenA \Mprod \matX}{\matY} = \dotprod{\matX}{\tenA^{\ha} \Mprod \matY}
    \]
\end{proposition}
\begin{proof}
    For every two oriented matrices $\matA,\matB\in\R^{p\times 1\times n}\cong\KM^p$
    \[
    \dotprod{\matA}{\matB} = \sum^p_{i=1}\dotprod{\matA_{i1:}}{\matB_{i1:}}
    \]
    so,
    \begin{align*}
        \dotprod{\tenA \Mprod \matX}{\matY} &= \sum^m_{i=1} \dotprod{(\tenA \Mprod \matX)_{i1:}}{\matY_{i1:}} \\
        &= \sum^m_{i=1}\dotprod{\sum^p_{j=1}\tenA_{ij}\Mprod\matX_{j1:}}{\matY_{i1:}} \\
        &= \sum^m_{i=1}\sum^p_{j=1}\dotprod{\tenA_{ij}\Mprod\matX_{j1:}}{\matY_{i1:}} \\
        &= \sum^p_{j=1}\sum^m_{i=1}\dotprod{\matX_{j1:}}{\tenA_{ij}^\sconj\Mprod\matY_{i1:}} \\
        &= \sum^p_{j=1} \dotprod{\matX_{j1:}}{\sum^m_{i=1}\tenA_{ij}^\sconj\Mprod\matY_{i1:}} \\
        &= \sum^p_{j=1} \dotprod{\matX_{j1:}}{(\tenA^\ha \Mprod \matY)_{j1:}} \\
        &= \dotprod{\matX}{\tenA^\ha \Mprod \matY}
    \end{align*}
\end{proof}

\end{document}